\documentclass[12pt,reqno]{amsart}
\usepackage{mathrsfs}
\usepackage{amsmath}
\usepackage{}
\usepackage{}
\usepackage{amsfonts}
\usepackage{a4wide}
\usepackage{amsmath,amsthm,amssymb,amscd}
\usepackage{latexsym}
\usepackage{hyperref}
\usepackage[numbers,sort&compress]{natbib}
\usepackage{hypernat}
\allowdisplaybreaks
\numberwithin{equation}{section}

\newtheorem{theorem}{Theorem}[section]
\newtheorem{proposition}[theorem]{Proposition}

\newtheorem{lemma}[theorem]{Lemma}

\theoremstyle{definition}

\newcommand{\va}{\varepsilon}

\newcommand{\ds}{\displaystyle}

\newcommand{\sik}{\sum_{i=1}^k}

\newcommand{\be}{\beta}
\newcommand{\al}{\alpha}

\def\r{\mathbb{R}}

\begin{document}

\title [Infinitely many solutions for a nonlocal equation]
{Positive or sign-changing solutions for a critical semilinear nonlocal equation}

\thanks{The authors sincerely thank Professor Shuangjie Peng for helpful discussions
and suggestions.}

\date{}
 \author{Wei Long \,\,and \,\,Jing Yang}

 \address{School of Mathematics and Statistics, Central China Normal
University, Wuhan, 430079, P. R. China\newline
 \noindent\quad College of Mathematics
and Information Science, Jiangxi Normal University, Nanchang,
Jiangxi 330022, P. R. China }

\email{ hopelw@126.com}

\address{School of Mathematics and Statistics, Central China
Normal University, Wuhan, 430079, P. R. China }

\email{ yyangecho@163.com}

\address{}

\begin{abstract}

We consider the following critical semilinear nonlocal equation involving the fractional Laplacian
$$
(-\Delta)^su=K(|x|)|u|^{2^*_s-2}u,\ \ \hbox{in}\ \  \r^N,
$$
where $K(|x|)$ is a positive radial function, $N>2+2s$, $0<s<1$, and
$2^*_s=\frac{2N}{N-2s}$. Under some asymptotic assumptions on $K(x)$ at an extreme point, we
show that this problem  has infinitely many non-radial positive  or sign-changing solutions.


\end{abstract}

\maketitle
\section{Introduction}
\ \ \ \ In this paper, we consider the following critical semilinear nonlocal equation involving the fractional Laplacian
\begin{equation} \label{eq}
(-\Delta)^su=K(x)|u|^{2^*_s-2}u,\ \ \hbox{in} \ \ \r^N
\end{equation}
with $N > 2+2s$, where $K(x)$ is a positive continuous potential, $0<s<1$,  $2^*_s=\frac{2N}{N-2s}$.
Here, the fractional Laplacian of a function $f:\r^N\rightarrow \r$ is
expressed by the formula
\begin{equation} \label{1.3}
(-\Delta)^sf(x)=C_{N,s}P.V.\int_{\r^N}\frac{f(x)-f(y)}{|x-y|^{N+2s}}dy=C_{N,s}\lim_{\delta \rightarrow
0}\int_{\r^N\setminus B_\delta (x)}\frac{f(x)-f(y)}{|x-y|^{N+2s}}dy,
\end{equation}
where
$C_{N,s}$ is some normalization constant.

Recently, a great attention has been focused on the study of problems involving
the fractional Laplacian, from a pure mathematical point of view as well as
from concrete applications, since this operator naturally arises in several areas such as physics, probability
and finance, see for instance \cite{a,b,cp,s}. In fact, the fractional Laplacian
can be understood as the infinitesimal generator of a stable L\'{e}vy process (see \cite{b}). The literature involving
the fractional Laplacian is really too wide to attempt any reasonable comprehensive
treatment in a single paper. We would just mention some very recent papers
which analyze fractional elliptic equations involving the critical Sobolev exponent (cf. \cite{bcps,sv,clo,li,lz,l,t,sv1,dpv,sz}).

It is well known, but not completely trivial, that $(-\Delta)^s$ reduces to the standard Laplacian $-\Delta$ as $s\to 1$.
Especially, when $s=1$, equation \eqref{eq} is reduced to the classical semilinear elliptic equation
\begin{equation}\label{1.3}
\left\{\begin{array}{ll}
       -\Delta u=K(x)|u|^{\frac{2}{N-2}}u,\ \ \ \ \ x\in\ \r^N\vspace{2mm}\\
         u\in D^{1,2}(\r^N),\\
    \end{array}\right.\\
\end{equation}
where $D^{1,2}(\r^N)$ denotes the completion of $C_0^{\infty}(\r^N)$ under the norm $\int_{\r^N}|\nabla u|^2$.

Our main interest in this paper is to investigate the multiplicity of solutions to equation \eqref{eq}. Before our study on this problem, we would like draw the reader's attention to some recent results on the multiplicity of positive solutions to equation \eqref{1.3}.
Amrosetti, Azorero and Peral \cite{aap}, and Cao, Noussair and Yan \cite{cny} proved the existence of two or many positive solutions if $K$ is a
perturbation of the constant, i.e.
\begin{equation*}
K= K_0+\varepsilon h(x), 0 < \varepsilon\ll 1.
\end{equation*}
In \cite{li1}, Li proved that \eqref{1.3} has infinitely many positive solutions if $K(x)$ is periodic,
while similar result was obtained in \cite{y} if $K(x)$ has a sequence of strict local maximum points
tending to infinity. In particular, in \cite{wy1},
Wei and Yan gave a very interesting result which says that equation \eqref{1.3} with $K(x)$ being radial has solutions with
large number of bumps near infinity and the energy of these solutions
can be arbitraily large. The reader can refer to \cite{cwy, wy2,wy3} for the existence of
infinitely solutions on semilinear equations involving critical and supcritical exponents.

 When $0<s<1,$ Chen and Zheng \cite{cz} studied  the following singularly perturbed problem
\begin{equation} \label{1.7}
\varepsilon^{2s}(-\Delta)^su+V(x) u=|u|^{p-1}u,\ \ \hbox{in} \ \ \r^N.
\end{equation}
They showed that when $N=1,2,3,$ $1<p<2^*_s-1$ and $\varepsilon$ is sufficiently small,
$\max\{\frac12,\frac n4\}<s<1$ and $V$ satisfies some smoothness and boundedness assumptions, equation
\eqref{1.7} has a nontrivial solution $u_{\epsilon}$
concentrated  to some single point as $\epsilon\rightarrow0.$ Very recently, in
\cite{dpw}, D\'{a}vila, del Pino and Wei generalized various existence results known for \eqref{1.7} with $s=1$ to the case of fractional Laplacian.
In \cite{lpy}, the authors gave a result which says that \eqref{1.7} with $V(x)$ being radial has solutions with
large number of bumps near infinity and the energy of this solutions
can be very large when $\va$ is fixed and $N\geq 2$.

Naturally, one would like know if the \textit{critical} equation \eqref{eq} has infinitely many non-radial solutions.
To the best of our knowledge, it seems that there is no answer for this question. The aim of this paper is to obtain infinitely many non-radial positive (sign-changing)
solutions for \eqref{eq} whose functional energy are arbitrarily large, under some assumptions that
 $K(x)=K(|x|)>0$ has a local maximum (minimum) at some point $r_0 > 0$.

 Letting
\begin{equation}\label{m}
\max\{2,N-2s-2\cdot\frac{(N-2s)^2}{N+2s}\}<m< N-2s,
\end{equation}
we assume that $0<K(|x|)\in C(\r^N)$
satisfies the following conditions at $r_0$
$$
(K):\,\,\, K(r)=K(r_0)-c_0|r-r_0|^m+O(|r-r_0|^{m+\theta}),\,\,\hbox{as}\,\,r\in (r_0-\delta,r_0+\delta),
$$
or
$$
(K'):\,\,\, K(r)=K(r_0)+c_0|r-r_0|^m+O(|r-r_0|^{m+\theta}),\,\,\hbox{as}\,\,r\in (r_0-\delta,r_0+\delta),
$$
where $c_0 > 0, \theta>0,\ \delta>0$ are some constants.

Without loss of generality, in what follows we assume that $K(r_0) = 1$.

Our main results in this paper can be stated as follows
\begin{theorem}\label{th1}
 Suppose that $N>2+2s,\,\,0<s<1$, and $K(r)$ satisfies ($K$). Then problem \eqref{eq} has
infinitely many non-radial positive solutions.
\end{theorem}

\begin{theorem}\label{th2}
 Suppose that $N>2+2s,\,\,0<s<1$, and $K(r)$ satisfies ($K'$). Then problem \eqref{eq} has
infinitely many non-radial sign-changing solutions.
\end{theorem}

For the sake of completeness, we firstly introduce basic theory on fraction Laplacian operator.

The nonlocal operator $(-\Delta)^s$ in $\r^N$ is defined on the Schwartz class through
the Fourier transform,
$$
\widehat{(-\Delta)^s f}(\xi) = (2\pi|\xi|)^{2s} \widehat{f(\xi)},
$$
or via the Riesz potential, see for example  (\cite{gt,p}) for the precise formula. As usual,
$\widehat{f}(\xi)$ denotes the Fourier Transform of $f$, $\widehat{f}(\xi) =\int_{\r^N}e^{-2\pi x\cdot\xi}f(x) dx.$
Observe that $s = 1$ corresponds to the standard local Laplacian. Since the fractional operator is nonlocal, L. Caffarelli and L. Silvestre showed in \cite{cs}
 that any fractional power of the Laplacian can be determined as an operator that maps a Dirichlet boundary condition
to a Neumann-type condition via an extension problem. To be more precise, consider the function
$u = u(x, y) : \r^N \times [0,\infty)\rightarrow\r,$ that solves the boundary value problem
\begin{equation} \label{2.2}
u(x,0)=f(x)
\end{equation}
and
\begin{equation} \label{2.3}
\Delta_x u+\frac{1-2s}{y}u_y+u_{yy}=0.
\end{equation}
Then, up to a multiplicative constant depending only on $s$,
$$
C(-\Delta)^sf=\lim_{y\rightarrow
0^+}-y^{1-2s}u_y=\frac{1}{2s}\lim_{y\rightarrow
0^+}\frac{u(x,y)-u(x,0)}{y^{2s}}.
$$
Thought the rest of this paper, the homogeneous fractional Sobolev space $D^s(\r^N) (0 < s < 1)$ is given by
$$
D^s(\r^N) =\Big \{u\in L^{\frac{2N}{N-2s}}(\r^N):\|u\|_{D^s(\r^N)}:=\Big(\int_{\r^N}|\xi|^{2s}|\widehat{u}(\xi)|^2\Big)^{\frac12}<+\infty\Big\}.
$$
 Note that $D^s(\r^N)$ is a Hilbert space equipped with an inner product
$$
\langle u,v\rangle_{D^s(\r^N)} = \int_{\r^N}|\xi|^{2s}\widehat{u}(\xi)\widehat{v}(\xi).
$$
We also define a fractional Laplace operator on the whole space, $(-\Delta)^s : D^s(\r^n)\rightarrow D^{-s}(\r^N)$ by
$$
\langle(-\Delta)^su, v\rangle_{D^{-s}(\r^N)} = \langle u, v\rangle_{D^s(\r^N)} ,
$$
where $D^{-s}(\r^N)$ is the dual of $D^s(\r^N)$. Then, one can easily check that if $u \in D^{2s}(\r^N)$, we have
$(-\Delta)^su \in L^2(\r^N)$ such that
$$
(-\Delta)^su = \mathfrak{F}^{-1}[|\xi|^{2s}\widehat{u}(\xi)]
$$
where $\mathfrak{F}^{-1}$ denotes the inverse Fourier transform so that we see for $u,v\in D^s(\r^n)$
$$
\langle u, v\rangle_{D^s(\r^N)} =\int_{\r^N}(-\Delta)^{\frac s2}u \cdot(-\Delta)^{\frac s2}v
$$
and assuming additionally $u \in  D^{2s}(\r^N), v\in L^2(\r^N)$, we can integrate by parts:
$$
\int_{\r^N}(-\Delta)^{\frac s2}u \cdot(-\Delta)^{\frac s2}v=\int_{\r^N}(-\Delta)^su\cdot v.
$$

In what remains of this paper, we will always mean that the equation
$$
(-\Delta)^s u=f\ \ \ in\ \r^N,
$$
is satisfied if
\begin{equation}\label{eq2}
u(x)=(-\Delta)^{-s}f(x):=\frac{1}{\gamma(N,s)}\int_{\r^N}\frac{f(y)}{|x-y|^{N-2s}}dy,
\end{equation}
as long as $f$ has enough decay for the integral to be well defined. The constant $\gamma(N,s)=\frac{\pi^{\frac N2}2^{2s}\Gamma(s)}{\Gamma(\frac N2-s)}.$
There are other definitions of $(-\Delta)^s u$, which are equivalent to \eqref{eq2} under suitable assumptions, see \cite{cs,stein}.

Lieb in 1983 \cite{l} (also see \cite{fl1,fl2}) established that the extremals correspond precisely
to functions of the form
\begin{equation}\label{2.3}
U_{\va,\xi}(x)=\alpha\Big(\frac{\va}{\va^2+|x-\xi|^2}\Big)^{\frac{N-2s}{2}},\ \ \alpha>0,
\end{equation}
which for a suitable choice
$$
\alpha=\alpha_{N,s}=2^{\frac{N-2s}{2}}\frac{\Gamma(\frac{N+2s}{2})}{\Gamma(\frac{N+2s}{2})}
$$
 solve the equation
\begin{equation}\label{2.4}
(-\Delta)^s u=u^{\frac{N+2s}{N-2s}},\ \ \ u>0,\ \ in\ \ \  \r^N.
\end{equation}

Very recently, J. D\'{A}vila, M. del Pino and Y. Sire \cite{dps} obtained the
non-degeneracy of the solutions for \eqref{2.4} in arbitrary
dimension $N>2s$.

So, we know that $U_{\va,\xi}(x)$ is non-degenerate in
$D^s(\r^{N})$. More precisely, define the functional
corresponding to \eqref{2.4} as follows
$$
f_0(u)=\frac{1}2\int_{\r^N}|(-\Delta)^{\frac s2}
u|^2-\frac1{2^*_s}\int_{\r^N}|u|^{2^*_s},\,\,\,u\in D^{2s}(\r^N).
$$
Then $f_0$ possesses a finite-dimensional
manifold $Z$ of least energy critical points, given by
$$
Z=\{U_{\va,\xi}:\va>0,\,\,\xi\in\r^{N}\}.
$$
Moreover,
$$
ker f''_0(z)=span_{\r}\left\{\frac{\partial U_{\va,\xi}}{\partial \xi_1},\cdots,
\frac{\partial U_{\va,\xi}}{\partial \xi_{N}},\,\frac{\partial U_{\va,\xi}}{\partial
\va}\right\},\,\,\,\forall\,\,U_{\xi,\va}\in Z.
$$

The following idea to prove our main results is essentially from \cite{wy1}.

 We will use the solution
$$
U_{\va,\xi}(x)=\alpha_{N,s}\Big(\frac{\va}{\va^2+|x-\xi|^2}\Big)^{\frac{N-2s}{2}},
$$
of
\begin{equation}\label{eq1}
 (-\Delta)^s u=u^{2^*_s-1},\ u>0,\,\,x \in \r^N,\,\,u(0)=\max\limits_{\r^N} u(x)
\end{equation}
to build up the approximate solutions for \eqref{eq}.
 From \cite{l,cl,fl1,fl2}, we see that when $s\in(0,1)$, the
unique ground solution of \eqref{eq1} decays like
$\frac{1}{|x|^{N-2s}}$ when $|x|\rightarrow \infty.$

Set
$$\nu=k^{\frac{N-2s}{N-2s-m}},
$$
to be the scaling parameter.
Using the transformation $u(x)\rightarrow\nu^{\frac{N-2s}{2}}u(\frac {x}{\nu})$, we find that \eqref{eq} becomes

$$
(-\Delta)^s u = K\big(\frac{x}{\nu}\big)u^{2^*_s-1},\ \ \ u>0,\ \ \ x\in\r^N.
$$

Throughout this paper, we always assume that
$$
r\in\Big[r_0\nu-\frac{1}{\nu^{\bar\theta}},r_0\nu+\frac{1}{\nu^{\bar\theta}}\Big],\ \ for\  some\  small \ \bar\theta >0,
$$
and
$$
\va_0\leq\va\leq\va_1, \ \ for\  some\  constants\  \va_1 >\va_0> 0.
$$

Set $x=(x',x'')\in  \r^2\times\r^{N-2}.$
Define
\begin{equation*} \begin{split}
\mathscr{H}=\Bigl\{u: &u\in D^{2s}(\r^N), u\ \mbox{is even in}\ x_j, j=2,\cdots,N,\\
&u(r\cos\theta, r\sin\theta,x'')=
u(r\cos(\theta+\frac{2i\pi}{k}),r\sin(\theta+\frac{2i\pi}{k}),x'')\Bigr\}.\\
\end{split} \end{equation*}
Write
$$
x^i=\big(r\cos\frac{2(i-1)\pi}{k},r\sin\frac{2(i-1)\pi}{k},0\big),\
\ i=1,\cdots,k,
$$
where $0$ is the zero vector in $\r^{N-2}$. Denote $$U_{\va,r}(x)=\sum_{i=1}^{k}U_{\va,x^i}(x).$$

To prove Theorem \ref{th1}, it suffices to verify the following result:
\begin{theorem}\label{th3}
Under the assumption of Theorem~\eqref{th1}, there is an
integer $k_0>0$, such that for any integer $k\geq k_0$, \eqref{eq}
has a solution $u_k$ of the form
$$
u_k=U_{\va,r}(x)+\omega_r
$$
where~~$\omega_r\in \mathscr{H}$,
~$r_k\in\Big[r_0\nu-\frac{1}{\nu^{\bar\theta}},r_0\nu+\frac{1}{\nu^{\bar\theta}}\Big]$ for some constants
$\bar\theta>0$
and
$\va_0\leq\va\leq\va_1$.

\end{theorem}

To consider the sign-changing solutions, for any integer $k>0$, we define
$$
\bar x^i=\big(r\cos\frac{(i-1)\pi}{k},r\sin\frac{(i-1)\pi}{k},0\big),\
\ i=1,\cdots,2k,
$$
write $$\bar U_{r,\va}(x)=\sum_{i=1}^{2k}(-1)^{i-1}U_{\bar x^i}(x).$$

Denote
\begin{equation*} \begin{split}
\mathscr{H'}=\Bigl\{u: &u\in D^{2s}(\r^N), u\ \mbox{is even in}\ x_j, j=2,\cdots,N,\\
&u(r\cos\theta, r\sin\theta,x'')=(-1)^i
u(r\cos(\theta+\frac{i\pi}{k}),r\sin(\theta+\frac{i\pi}{k}),x'')\Bigr\}.\\
\end{split} \end{equation*}

To prove Theorem~ \ref{th2}, we only need to show that:

\begin{theorem}\label{th4}
Under the assumption of Theorem~\eqref{th1}, there is an
integer $k_0>0$, such that for any integer $k\geq k_0$, \eqref{eq}
has a solution $u_k$ of the form
$$
\bar u_k=\bar U_{\va,r}(x)+\bar\omega_r
$$
where~~$\bar\omega_r\in \mathscr{H'}$,
~$r_k\in\Big[r_0\nu-\frac{1}{\nu^{\bar\theta}},r_0\nu+\frac{1}{\nu^{\bar\theta}}\Big]$ for some constants
$\bar\theta>0$
and
$\va_0\leq\va\leq\va_1$.

\end{theorem}
The rest of the paper is organized as follows. In Sect.2, we will carry out a
reduction procedure. We prove our main result in Section 3. Finally, in Appendix, some basic estimates
and an energy expansion for the functional
corresponding to problem \eqref{eq} will be established.

\section{Finite-dimension Reduction}
 In this section, we perform a finite-dimensional reduction. Let
\begin{equation}\label{3.1}
\|u\|_{*}=\sup_{x\in\r^{N}}\Big(\sum_{i=1}^{k}\frac{1}{(1+|x-x^i|)^{\frac{N-2s}{2}+\tau}}\Big)^{-1}|u(x)|
\end{equation}
and
\begin{equation}\label{3.2}
\|f\|_{**}=\sup_{x\in\r^{N}}\Big(\sum_{i=1}^{n}\frac{1}{(1+|x-x^i|)^{\frac{N+2s}{2}+\tau}}\Big)^{-1}|f(x)|,
\end{equation}
where $\tau=\frac{N-2s-m}{N-2s}.$
Write
$$
Z_{i,1}=\frac{\partial U_{x^i,\va}}{\partial r},\ \ \ Z_{i,2}=\frac{\partial U_{x^i,\va}}{\partial \va}.
$$
Consider
 \begin{equation}\label{3.3}
\left\{%
\begin{array}{ll}
    (-\Delta)^s \varphi_{k}-(2^*_s-1)K\big(\frac{|x|}{\nu}\big)\ds U_{r,\va}^{2^*_s-2} \varphi_{k}=
    H_{k}+\ds\sum_{l=1}^{2}c_{l}\ds\sum_{i=1}^{k} U_{x^{i},\va}^{2^*_s-2} Z_{i,l}, ~~&  x\in {\r}^{N}, \vspace{0.2cm}\\
   \varphi_{k}\in \mathscr{H},\vspace{0.2cm}\\
  \Big \langle \ds U_{x^{i},\va}^{2^*_s-2} Z_{i,l},\varphi_{k}\Big\rangle=0,\ \ i=1,...,k,l=1,2,
\end{array}%
\right.
\end{equation}
for some numbers $c_l$, where
$
\langle u,v\rangle= \int_{\r^N} uv.
$

 \begin{lemma}\label{lm3.1}
 Suppose that $\varphi_{k}$ solves \eqref{3.3} for $H=H_{k}$. If $\|H_{k}\|_{**}$ goes to zero as $k$ goes to infinity, so does $\|\varphi_{k}\|_{*}$.
 \end{lemma}
\begin{proof}
We will  argue by an indirect method. Suppose to the contrary that
there exist $k\rightarrow
+\infty,H=H_{k},\va_{k}\in[\va_{0},\va_{1}],r_k\in\left[r_{0}\nu-\nu^{-\bar\theta},\,r_{0}\nu+\nu^{-\bar\theta}\right]$
and $\varphi_{k}$ solving \eqref{3.3} for
$H=H_{k},\va=\va_{k},r=r_{k}$ with
$\|H_{k}\|_{**}\rightarrow0$ and $\|\varphi_{k}\|_{*}\geq c>0.$ We
may assume that $\|\varphi_{k}\|_{*}=1$. For simplicity, we drop the
subscript $k$.
By \eqref{eq2}, we can rewrite \eqref{3.3} as
 \begin{equation}\label{3.4}
\begin{array}{ll}
\varphi(x)= &\frac{1}{\gamma(N,s)}(2^*_s-1)\ds\int_{\r^N} \frac{1}{|y-x|^{N-2s}}K
\Big(\frac{|y|}{\nu}\Big)U_{r,\va}^{2^*_s-2}(y)\varphi(y)dy\vspace{0.2cm}\\
&+
\frac{1}{\gamma(N,s)}\ds\int_{\r^N}\frac{1}{|y-x|^{N-2s}}\Big( H(y)+\ds\sum_{l=1}^{2}c_{l}\ds\sum_{i=1}^{k}U_{x_{i},\va}^{2^*_s-2}(y)Z_{i,l}(y)\Big)dy,
 \end{array}
\end{equation}
Analogously to Lemma \ref{lmA.3},  we have
\begin{equation}\label{2.5}
\begin{array}{ll}
&\left|\ds\int_{\r^N} \frac{1}{|y-x|^{N-2s}}K
\Big(\frac{|y|}{\nu}\Big)U_{r,\va}^{2^*_s-2}(y)\varphi(y)dy\right|\vspace{0.2cm}\\
&\leq C\|\varphi\|_{*}\ds\int_{\r^N}\frac{1}{|y-x|^{N-2s}}U_{r,\va}^{2^*_s-2}(y)
\sum_{i=1}^{k}\frac{1}{(1+|y-x^{i}|)^{\frac{N-2s}{2}+\tau}}dy\vspace{0.2cm}\\
&\leq  C\|\varphi\|_{*}\ds\sum_{i=1}^{k}\frac{1}{(1+|x-x^i|)^{\frac{N-2s}{2}+\tau+\theta}}.
 \end{array}
\end{equation}

By Lemma \ref{lmA.2}, we get
\begin{equation}\label{3.6}
\begin{array}{ll}
\left|\ds\int_{\r^N} \frac{1}{|x-y|^{N-2s}}H(y)dy\right|
&\leq C\|H\|_{**}\ds\int_{\r^N}\frac{1}{|x-y|^{N-2s}}
\sum_{i=1}^{k}\frac{1}{(1+|y-x^{i}|)^{\frac{N+2s}{2}+\tau}}dy\vspace{0.2cm}\\
&\leq  C\|H\|_{**}\ds\sum_{i=1}^{k}\frac{1}{(1+|x-x^{i}|)^{\frac{N-2s}{2}+\tau}}
 \end{array}
\end{equation}
and
\begin{equation}\label{3.7}
\begin{array}{ll}
&\left|\ds\int_{\r^N}\frac{1}{|x-y|^{N-2s}}\ds\sum_{i=1}^{k}U_{x^i,\va}^{2^*_s-2}(y)Z_{i,l}(y)dy\right|\vspace{0.2cm}\\
&\leq C\ds\int_{\r^N}\ds\frac{1}{|x-y|^{N-2s}}
\ds\sum_{i=1}^{k}U_{x^{i},\va}^{2^*_s-1}(y)dy\vspace{0.2cm}\\
&\leq C\ds\int_{\r^N}\ds\frac{1}{|x-y|^{N-2s}}\ds\sum_{i=1}^{k}\frac{1}{(1+|y-x^{i}|)^{N+2s}}dy\vspace{0.2cm}\\
&\leq  C\ds\sum_{i=1}^{k}\frac{1}{(1+|x-x^{i}|)^{N}}\vspace{0.2cm}\\
&\leq  C\ds\sum_{i=1}^{k}\frac{1}{(1+|x-x^{i}|)^{\frac{N-2s}{2}+\tau}}.
 \end{array}
\end{equation}
Next, we estimate $c_{l},l=1,2.$ Multiplying \eqref{3.3} by $Z_{1,t}(t=1,2)$ and integrating, we see that $c_{l}$ satisfies
\begin{equation}\label{3.8}
\sum_{l=1}^{2}c_{l}\sum_{i=1}^{k}\Big\langle U_{x^{i},\va}^{2^*_s-2}Z_{i,l},Z_{1,t}\Big\rangle
=
\Big\langle(-\Delta)^s\varphi
-(2^*_s-1) K \big(\frac{|y|}{\nu}\big)\ds U_{r,\va}^{2^*_s-2}\varphi,Z_{1,t}\Big\rangle
-\Big\langle H, Z_{1,t}\Big\rangle.
 \end{equation}
It follows from Lemma \ref{lmA.1} that
 $$
 \begin{array}{ll}
|\left\langle H, Z_{1,t}\right\rangle|&\leq C\|H\|_{**}\ds\int_{\r^N}\frac{1}{(1+|y-x^{1}|)^{N-2s}}
\sum_{i=1}^{k}\frac{1}{(1+|y-x^{i}|)^{\frac{N+2s}{2}+\tau}}\\
&\leq C \|H\|_{**}.
 \end{array}
 $$
  Then
 \begin{equation}\label{3.9}
\begin{array}{ll}
&\Big\langle(-\Delta)^s\varphi
-(2^*_s-1)K\big(\frac{|y|}{\nu}\big)\ds U_{r,\va}^{2^*_s-2}\varphi,Z_{1,t}\Big\rangle\vspace{0.15cm}\\
&=\Big\langle(-\Delta)^s Z_{1,t}
-(2^*_s-1)K\big(\frac{|y|}{\nu}\big)\ds U_{r,\va}^{2^*_s-2}Z_{1,t},\varphi\Big\rangle
\vspace{0.15cm}\\
&=\Big\langle (2^*_s-1)\ds U_{x^{1},\va}^{2^*_s-2}Z_{1,t}
-(2^*_s-1)K\big(\frac{|y|}{\nu}\big)\ds U_{r,\va}^{2^*_s-2}Z_{1,t},\varphi\Big\rangle\vspace{0.15cm}\\
&=C\|\varphi\|_*\ds \int \Big|K\big(\frac{|y|}{\nu}\big)-1\Big| U_{r,\va}^{2^*_s-2}
\frac{1}{(1+|y-x^1|)^{N-2s}}\sik\frac{1}{(1+|y-x^i|)^{\frac{N-2s}2+\tau}}.
 \end{array}
 \end{equation}
If $\Big||y|-\nu r_0\Big|\geq\sqrt{\nu}$, then
 $$
\big||y|-x^1\big|\geq \big||y|-\nu r_0\big|-\big||x^1|-\nu r_0\big|\geq \sqrt{\nu}-\frac{1}{\nu^{\bar\theta}}\geq\frac12\sqrt{\nu}.
 $$
Computing as Lemma~\ref{lmA.3}, we see
  \begin{eqnarray*}
&& \ds\int_{\big||y|-\nu r_0\big|\geq \sqrt{\nu}}\Big|K\big(\frac{|y|}{\nu}\big)-1\Big| U_{r,\va}^{2^*_s-2}
\frac{1}{(1+|y-x^1|)^{N-2s}}\sik\frac{1}{(1+|y-x^i|)^{\frac{N-2s}2+\tau}}dy\\
&\leq & \frac{C}{\nu^\sigma}\ds\int_{\r^N} U_{r,\va}^{2^*_s-2}
\frac{1}{(1+|y-x^1|)^{N-2s}}\sik\frac{1}{(1+|y-x^i|)^{\frac{N-2s}2+\tau-2\sigma}}dy\\
&\leq & \frac{C}{\nu^\sigma},
\end{eqnarray*}
where $\sigma$ is a small constant.
On the other hand, as Lemma~\ref{lmA.3}, we have
 \begin{eqnarray*}
&& \ds\int_{\big||y|-\nu r_0\big|\leq \sqrt{\nu}}\Big|K\big(\frac{|y|}{\nu}\big)-1\Big| U_{r,\va}^{2^*_s-2}
\frac{1}{(1+|y-x^1|)^{N-2s}}\sik\frac{1}{(1+|y-x^i|)^{\frac{N-2s}2+\tau}}\\
&\leq & \Big(\frac{C}{\sqrt{\nu}}\Big)^m\ds\int_{\r^N} U_{r,\va}^{2^*_s-2}
\frac{1}{(1+|y-x^1|)^{N-2s}}\sik\frac{1}{(1+|y-x^i|)^{\frac{N-2s}2+\tau}}\\
&\leq & \frac{C}{\sqrt{\nu}}.
\end{eqnarray*}
So, together with \eqref{3.9}, we obtain
\begin{equation}\label{3.10}
\Big\langle(-\Delta)^s\varphi
-(2^*_s-1)K\big(\frac{|y|}{\nu}\big)\ds U_{r,\va}^{2^*_s-2}\varphi,Z_{1,t}\Big\rangle=o(1)\|\varphi\|_*.
\end{equation}
But
 \begin{eqnarray*}
&&\sum_{i=1}^{k}\Big\langle U_{x^{i},\va}^{2^*_s-2}Z_{i,l},Z_{1,t}\Big\rangle\\
&\leq& \sum_{i=1}^{k}C\ds\int_{\r^N} \Big(\frac{1}{(1+|x-x^1|)^{N-2s}}\Big)^{2^*_s-1}\frac{1}{(1+|x-x^i|)^{N-2s}}dx\\
&\leq& C.
\end{eqnarray*}
 Hence it follows from \eqref{3.8} that
 \begin{equation}\label{3.11}
 c_{l}=O\left(\frac{1}{\nu^\sigma}\|\varphi\|_{*}+\|H\|_{**}\right)=o(1).
 \end{equation}
 So,
  \begin{equation}\label{3.12}
 \|\varphi\|_{*}\leq \Big(\|H\|_{**}+\frac{\sum\limits_{i=1}^{k}\frac{1}{(1+|y-x^{i}|)^{\frac{N-2s}{2}+\tau+\theta}}}
{\sum\limits_{i=1}^{k}\frac{1}{(1+|y-x^{i}|)^{\frac{N-2s}{2}+\tau}}}\Big).
 \end{equation}
  Since $\|\varphi\|_{*}=1$, we obtain from \eqref{3.12} that there is $R>0$ such that
 \begin{equation}\label{3.13}
 \|\varphi(x)\|_{L^\infty\left(B_{R}(x^{i})\right)}\geq a>0,
 \end{equation}
 for some $i$.  But $\tilde{\varphi}(x)=\varphi(x-x^i)$ converges uniformly in any compact set to a solution $u$ of
 \begin{equation}\label{3.14}
(-\Delta)^s u-(2^*_s-1)U_{0,\va}^{2^*_s-2}u=0,\,\,\,\text{in}\,\,\,
\r^{N},
 \end{equation}
 for some $\va\in[\va_{0},\va_{1}]$ and $u$ is perpendicular to the kernel of \eqref{3.14}. So $u=0$. This is a contradiction to \eqref{3.12}.
\end{proof}
\begin{proposition}\label{prop3.2} There exists $k_0>0$ and a constant $C>0,$ independent of $k,$ such that for all $k\geq k_0,$
 and all $H\in L^\infty(\r^N)$, problem \eqref{3.3} has a unique solution $\varphi\equiv l_k(H)$. Besides
  \begin{equation}\label{3.15}
 \|L_{k}(H)\|_{*}\leq C\|H\|_{**},
 \end{equation}
   \begin{equation}\label{3.16}
|c_{l}|\leq C\|H\|_{**}.
 \end{equation}
\end{proposition}
\begin{proof}
Following from \cite{dfm}, let us consider the space
$$
\mathbb{H}=\Big\{\varphi\in D^{2s}(\r^N)|\Big \langle \ds U_{x^{i},\va}^{2^*_s-2}Z_{i,l},\varphi\Big\rangle=0,\ \ i=1,...,k,l=1,2,\Big\}
$$
endowed with usual inner product $[\varphi,\psi]=\int_{\r^N}(-\Delta)^{\frac s2}\varphi(-\Delta)^{\frac s2}\psi$. Problem \eqref{3.3} expressed
in weak form is equivalent to that of finding a $\varphi\in \mathbb{H}$ such that
$$
[\varphi,\psi]=\langle(2^*_s-1)K\big(\frac{|x|}{\nu}\big)\ds U_{r,\va}^{2^*_s-2} \varphi+ H,\psi\rangle,\ \ \ \ \forall \psi\in \mathbb{H}.
$$
With the aid of Riesz's representation theorem, this equation gets rewritten in $\mathbb{H}$ in the operational form
\begin{equation}\label{3.17}
\varphi=T_{k}(\varphi)+\bar H
\end{equation}
with certain $\bar H\in \mathbb{H}$ which depends linearly in $H$ and where $T_{k}$ is a compact operator in $\mathbb{H}$. Fredholm's
alternative guarantees unique solvability of this problem for any $H$ provided that the homogeneous equation
$$
\varphi=T_k(\varphi)
$$
has only the zero solution in $\mathbb{H}$. Let us observe that this last equation is equivalent to
 \begin{equation}\label{3.18}
\left\{%
\begin{array}{ll}
    (-\Delta)^s \varphi-(2^*_s-1)K\big(\frac{|x|}{\nu}\big)\ds U_{r,\va}^{2^*_s-2} \varphi=
    \ds\sum_{l=1}^{2}c_{l}\ds\sum_{i=1}^{k} U_{x^{i},\va}^{2^*_s-2} Z_{i,l}, ~~&  x\in {\r}^{N}, \vspace{0.2cm}\\
   \varphi\in \mathscr{H},\vspace{0.2cm}\\
  \Big \langle \ds U_{x^{i},\va}^{2^*_s-2} Z_{i,l},\varphi\Big\rangle=0,\ \ i=1,...,k,l=1,2,
\end{array}%
\right.
\end{equation}
for certain constants $c_l$. Assume it has a nontrivial solution $\varphi=\varphi_k,$ which with
no loss of generality may be taken so that $\|\varphi_k\|_*=1.$ But this makes the Lemma \ref{lm3.1},
so that necessarily $\|\varphi_k\|_*\rightarrow 0.$ This is certainly a contradiction that proves that
this equation only has the trivial solution in $\mathbb{H}$. We conclude then that for each $H$,
problem \eqref{3.3} admits a unique solution. We check that
$$
\|\varphi\|_* \leq \|H\|_{**}.
$$
We assume again the opposite. In doing so, we find a sequence $H_k$ with $\|H\|_{**}=o(1)$ and solution
$\varphi_k\in \mathbb{H}$ of problem \eqref{3.3} with $\|\varphi\|_*=1.$ Again this makes the Lemma \ref{lm3.1} and a
contradiction has been found. This proves estimates \eqref{3.15}. Estimate \eqref{3.16} follows from this and relation
\eqref{3.11}. This concludes this proof of the proposition.
\end{proof}

Now, we consider
 \begin{equation}\label{3.19}
\left\{%
\begin{array}{ll}
    (-\Delta)^s \varphi-(2^*_s-1)K\big(\frac{|x|}{\nu}\big)\ds U_{r,\va}^{2^*_s-2} \varphi=
    N(\varphi)+l_k+\ds\sum_{l=1}^{2}c_{l}\ds\sum_{i=1}^{k} U_{x^{i},\va}^{2^*_s-2} Z_{i,l}, ~~&  x\in {\r}^{N}, \vspace{0.2cm}\\
   \varphi\in \mathscr{H},\vspace{0.2cm}\\
  \Big \langle \ds U_{x^{i},\va}^{2^*_s-2}Z_{i,l},\varphi\Big\rangle=0,\ \ i=1,...,k,l=1,2,
\end{array}%
\right.
\end{equation}
where
$$
N(\varphi)=K\big(\frac{|x|}{\nu}\big)\Big(\big(U_{r,\va}+\varphi\big)^{2^*_s-1}-U_{r,\va}^{2^*_s-1}-(2^*_s-1)U_{r,\va}^{2^*_s-2}\varphi\Big)
$$
and
$$
l_k=K\big(\frac{|x|}{\nu}\big)U_{r,\va}^{2^*_s-1}-\sik U_{x^i,\va}^{2^*_s-1}.
$$
Next, we estimate $N(\varphi)$ and $l_k.$
\begin{lemma}\label{lm3.3} We obtain
$$
\|N(\varphi)\|_{**}\leq C\|\varphi\|_*^{\min\{2^*_s-1,2\}}.
$$
\end{lemma}
\begin{proof}
Firstly, we deal with the case $2^*_s\leq3.$ Since
$$
\sik \frac{1}{(1+|x-x^i|)^{\frac{N-2s-m}{N-2s}}}\leq C+\sum_{i=2}^k\frac{C}{|x^1-x^i|^{\frac{N-2s-m}{N-2s}}}\leq C,
$$
we find by H\"{o}lder inequality,
 \begin{eqnarray*}
 |N(\varphi)|&\leq&C \|\varphi\|_{*}^{2^*_s-1}
 \Big(\sum_{i=1}^{k}\frac{1}{(1+|x-x^{i}|)^{\frac{N-2s}{2}+\tau}}\Big)^{2^*_s-1} \\
&\leq&
 C \|\varphi\|_{*}^{2^*_s-1}\sum_{i=1}^{k}\frac{1}{(1+|x-x^{i}|)^{\frac{N+2s}{2}+\tau}}
 \Big(\sum_{i=1}^{k}\frac{1}{(1+|x-x^{i}|)^{\tau}}\Big)^{\frac{4s}{N-2s}}\\
 &\leq&  C\|\varphi\|_{*}^{2^*_s-1}\sum_{i=1}^{k}\frac{1}{(1+|y-x^{i}|)^{\frac{N+2s}{2}+\tau}}.
 \end{eqnarray*}
Using  the same argument, for the case $2^*_s>3$,  we also can obtain that
\begin{eqnarray*}
 |N(\varphi)|&\leq&C \|\varphi\|_{*}^2\Big(\sum_{i=1}^{k}\frac{1}{(1+|x-x^{i}|)^{N-2s}}\Big)^{2^*_s-3}
 \Big(\sum_{j=1}^{k}\frac{1}{(1+|x-x^{j}|)^{\frac{N-2s}{2}+\tau}}\Big)^2 \\
 &&+\|\varphi\|_{*}^{2^*_s-1}\Big(\sum_{i=1}^{k}\frac{1}{(1+|x-x^{i}|)^{\frac{N-2s}{2}+\tau}}\Big)^{2^*_s-1}\\
&\leq&
 C (\|\varphi\|_{*}^{2^*_s-1}+\|\varphi\|_{*}^2)\Big(\sum_{i=1}^{k}\frac{1}{(1+|x-x^{i}|)^{\frac{N-2s}{2}+\tau}}\Big)^{2^*_s-1}\\
 &\leq&  C\|\varphi\|_{*}^{2}\sum_{i=1}^{k}\frac{1}{(1+|x-x^{i}|)^{\frac{N+2s}{2}+\tau}}.
 \end{eqnarray*}

\end{proof}

\begin{lemma}\label{lm3.4} Assume that $||x^1|-\nu r_0|\leq \frac{1}{\nu^{\bar\theta}}$, where $\bar\theta>0$ is a fixed small constant. Then
 there is a small $\sigma>0,$ such that
 $$
 \|l_k\|_{**}\leq C\Big(\frac{1}{\nu}\Big)^{\frac m2+\sigma}.
 $$
\end{lemma}
\begin{proof}
Define
$$
\Omega_{i}=\left\{x:x=(x',x'')\in \r^{2}\times\r^{N-2},\left\langle \frac{x'}{|x'|},\frac{x^{i}}{|x^{i}|}\right\rangle\geq \cos\frac{\pi}{k}\right\}.
$$
We have
\begin{eqnarray*}
 l_k&=&K\Big(\frac{|x|}{\nu}\Big)\Big(U_{r,\va}^{2^*_s-1}-\sik U_{x^i,\va}^{2^*_s-1}\Big)+\sik U_{x^i,\va}^{2^*_s-1}\Big(K\Big(\frac{|x|}{\nu}\Big)-1\Big)\\
 &:=&J_1+J_2.
 \end{eqnarray*}
 By symmetry, we can assume that $x\in\Omega_{1}$. Then
\begin{equation*}
|x-x^{i}|\geq|x-x^1|,\,\,\forall \ x\in\Omega_1.
\end{equation*}
In order to estimate $J_{1},$ we have
\begin{equation}\label{3.20}
\begin{array}{ll}
|J_{1}|&\leq C\Big(\big(\sum\limits_{i=1}^k \ds\frac{1}{(1+|x-x^i|)^{N-2s}}\big)^{2^*_s-1}-
\sum\limits_{i=1}^k\big(\ds\frac{1}{(1+|x-x^i|)^{N-2s}}\big)^{2^*_s-1}\Big)\\
&\leq C \ds\frac{1}{(1+|x-x^1|)^{4s}}\sum\limits_{i=2}^k \ds\frac{1}{(1+|x-x^i|)^{N-2s}}+C\Big(\sum\limits_{i=2}^k \ds\frac{1}{(1+|x-x^i|)^{N-2s}}\Big)^{2^*_s-1}.
 \end{array}
\end{equation}
By Lemma \ref{lmA.1}, taking any $1<\alpha\leq \frac{N+2s}{2}$, we obtain that for any $x\in\Omega_{1},$
\begin{equation}\label{3.21}
\begin{array}{ll}
&\ds\frac{1}{(1+|x-x^1|)^{4s}}\ds\sum_{i=2}^{k}\frac{1}{(1+|x-x^{i}|)^{N-2s}}
\leq\ds\frac{1}{(1+|x-x^1|)^{\frac{N+2s}{2}}}\ds\sum_{i=2}^{k}\frac{1}{(1+|x-x^{i}|)^{\frac{N+2s}{2}}}\vspace{0.2cm}\\
&\leq
C\ds\sum_{i=2}^{k}\left[\ds\frac{1}{(1+|x-x^1|)^{N+2s-\alpha}}+\ds\frac{1}{(1+|x-x^{i}|)^{N+2s-\alpha}}\right]
\frac{1}{|x^i-x^1|^{\alpha}}\vspace{0.2cm}\\
&\leq C\ds\frac{1}{(1+|x-x^1|)^{N+2s-\alpha}}\sum_{i=2}^{k}\frac{1}{|x^i-x^1|^{\alpha}}\vspace{0.2cm}\\
&\leq \ds\frac{C}{(1+|x-x^i|)^{N+2s-\alpha}}\left(\frac{k}{\nu}\right)^{\alpha}.
 \end{array}
 \end{equation}
 We can choose $\alpha>\frac{N-2s}{2}$ with $N+2s-\alpha\geq \frac{N+2s}{2}+\tau $ since
 $m\geq N-2s-2s\frac{(N-2s)^2}{N+2s}>N-2s-2s(N-2s).$ Hence
 \begin{eqnarray*}
\ds\frac{1}{(1+|x-x^{1}|)^{4s}}\sum_{i=2}^{k}\frac{1}{(1+|x-x^{i}|)^{N-2s}}
\leq
\ds\frac{C}{(1+|x-x^{1}|)^{\frac{N+2s}{2}+\tau}}\left(\frac{1}{\nu}\right)^{\frac{m}{2}+\epsilon}.
\end{eqnarray*}
On the other hand, for $x\in \Omega_{1}$, by Lemma \ref{lmA.1} again, we get
\begin{eqnarray*}
&&\ds\frac{1}{(1+|x-x^{i}|)^{N-2s}}\\
&\leq& \frac{1}{(1+|x-x^{1}|)^{\frac{N-2s}{2}}}\frac{1}{(1+|x-x^{i}|)^{\frac{N-2s}{2}}}
\\
&\leq&\frac{C}{|x^{i}-x^{1}|^{\frac{N-2s}{2}-\frac{N-2s}{N+2s}\tau}}
\Big( \frac{1}{(1+|x-x^{1}|)^{\frac{N-2s}{2}+\frac{N-2s}{N+2s}\tau}}+\frac{1}{(1+|x-x^{i}|)^{\frac{N-2s}{2}+\frac{N-2s}{N+2s}\tau}}\Big)
\\
&\leq&\frac{C}{|x^{i}-x^{1}|^{\frac{N-2s}{2}-\frac{N-2s}{N+2s}\tau}}
 \frac{1}{(1+|x-x^{1}|)^{\frac{N-2s}{2}+\frac{N-2s}{N+2s}\tau}}.
\end{eqnarray*}
When $m>2,\ \tau=\frac{N-2s-m}{N-2s}$, we can obtain $\frac{N-2s}{2}-\frac{N-2s}{N+2s}\tau>1.$ Thus
$$
\sum_{i=2}^k\frac{1}{(1+|x-x^i|)^{N-2s}}\leq C\Big(\frac k\nu\Big)^{\frac{N-2s}{2}-\frac{N-2s}{N+2s}\tau}\frac{1}{(1+|x-x^{1}|)^{\frac{N-2s}{2}+\frac{N-2s}{N+2s}\tau}},
$$
which gives
$$
\Big(\sum_{i=2}^k\frac{1}{(1+|x-x^i|)^{N-2s}}\Big)^{2^*_s-1}\leq C\Big(\frac k\nu\Big)^{\frac{N+2s}{2}-\tau}\frac{1}{(1+|x-x^{1}|)^{\frac{N+2s}{2}+\tau}}.
$$
So,
$$
\|J_1\|_{**}\leq C\Big(\frac 1\nu\Big)^{\frac m2+\epsilon}.
$$
Next, we estimate $J_2.$ For any $x\in\Omega_{1}$ and $i=2,\cdots,k$, applying Lemma \ref{lmA.1}, we have
\begin{eqnarray*}
U_{x^i,\va}^{2^*_s-1}(x)&\leq &
C\frac{1}{(1+|x-x^1|)^{\frac{N+2s}{2}}}\frac{1}{(1+|x-x^i|)^{\frac{N+2s}{2}}}\\
&\leq& C\Big(\frac{1}{(1+|x-x^1|)^{\frac{N+2s}{2}+\tau}}+\frac{1}{(1+|x-x^i|)^{\frac{N+2s}{2}+\tau}}\Big)\frac{1}{|x^i-x^1|^{\frac{N+2s}{2}-\tau}}\\
&\leq& C
\ds\frac{1}{(1+|x-x^1|)^{\frac{N+2s}{2}+\tau}}\frac{1}{|x^i-x^1|^{\frac{N+2s}{2}-\tau}},
\end{eqnarray*}
which implies that
\begin{eqnarray*}
\left|\sum_{i=2}^k\Big(K\big(\frac{|x|}{\nu}\big)-1\Big)U_{x^i,\va}^{2^*_s-1}\right|&\leq &
C\frac{1}{(1+|x-x^i|)^{\frac{N+2s}{2}+\tau}}\sum_{i=1}^{k}\frac{1}{|x^i-x^1|^{\frac{N+2s}{2}-\tau}}\\
&\leq&C
\ds\frac{1}{(1+|x-x^1|)^{\frac{N+2s}{2}+\tau}}\Big(\frac{k}{\nu}\Big)^{\frac{N+2s}{2}-\tau}.
\end{eqnarray*}
For $x\in\Omega_1$ and $\big||x|-\nu r_0\big|\geq \delta \nu$, where $\delta>0$ is a fixed constant, then
$$
\big||x|-|x^1|\big|\geq\big||x|-\nu r_0\big|-\big||x^1|-\nu r_0\big|\geq \frac{\delta \nu}{2}.
$$
As a result,
\begin{eqnarray*}
\Big|\Big(K\Big(\frac{|x|}{\nu}\Big) -1\Big)U_{x^1,\va}^{2^*_s-1}\Big|
&\leq&\frac{C}{(1+|x-x^1|)^{N+2s}}\\
&\leq&\frac{C}{(1+|x-x^1|)^{\frac{N+2s}{2}+\tau}}\left(\frac{1}{\nu}\right)^{\frac{N+2s}{2}-\tau}.
\end{eqnarray*}
If $x\in\Omega_1$ and $\big||x|-\nu r_0\big|\leq\delta\nu$, then
\begin{eqnarray*}
\Big|K\Big(\frac{|x|}{\nu}\Big) -1\Big|&\leq& C\Big|\frac{|x|}{\nu}-r_0\Big|^{m}\\
&\leq&
\frac{C}{\nu^{m}}\Big(\big||x|-|x^1|\big|^{m}+\big||x^1|-\nu r_0\big|^{m}\Big)\\
&\leq&\frac{C}{\nu^{m}}\big||x|-|x^1|\big|^{m}+\frac{C}{\nu^{m+\bar{\theta}}}
\end{eqnarray*}
and
$$
\big||x|-|x^1|\big|\leq \big||x|-\nu r_{0}\big|+\left|\nu r_{0}-|x^1|\right|
\leq 2\delta\nu.
$$
As a result,
\begin{eqnarray*}
&&\frac{\big||x|-|x^1|\big|^{m}}{\nu^{m}}\frac{1}{(1+|x-x^1|)^{N+2s}}\\
&\leq&\frac{C}{\nu^{\frac{m}{2}+\epsilon}}\frac{\big||x|-|x^1|\big|^{\frac{m}{2}+\epsilon}}{(1+|x-x^1|)^{N+2s}}
\\
&\leq&\frac{C}{\nu^{\frac{m}{2}+\epsilon}}\frac{1}{(1+|x-x^1|)^{\frac{N+2s}{2}+\tau}}
\frac{1}{(1+|x-x^1|)^{\frac{N+2s}{2}-\tau-\frac{m}{2}-\epsilon}}\\
&\leq&\frac{C}{\nu^{\frac{m}{2}+\epsilon}}\frac{1}{(1+|x-x^i|)^{\frac{N+2s}{2}+\tau}},
\end{eqnarray*}
since  $\frac{N+2s}{2}-\tau-\frac{m}{2}-\epsilon\geq\frac{N+2s}{2}-\tau-\frac{N-2s}{2}-\epsilon>0.$
Thus, we obtain
\begin{eqnarray*}
\Big|\Big(K\big(\frac{|x|}{\nu}\big) -1\Big)U_{x^i,\va}^{2^*_s-1}\Big|
\leq\frac{C}{\nu^{\frac{m}{2}+\epsilon}}\frac{1}{(1+|x-x^i|)^{\frac{N+2s}{2}+\tau}},\,\,\,
\left||x|-\nu r_0\right|\leq\delta\nu.
\end{eqnarray*}
Hence,
$$
\|J_{2}\|_{**}\leq C\left(\frac{1}{\nu}\right)^{\frac{m}{2}+\epsilon}.
$$
Therefore,
$$
\|l_{k}\|_{**}\leq
C\left(\frac{1}{\nu}\right)^{\frac{m}{2}+\epsilon}.
$$
\end{proof}
From Lemmas \ref{lm3.3} and \ref{lm3.4}, we have
\begin{proposition}\label{prop3.3} There is an integer $k_0>0,$ such that for each $k\geq k_0,\ \va_0\leq\va\leq \va_1,\ |r-\nu r_0|\leq \frac{1}{\nu^{\bar\theta}},$ where $\bar\theta>0$ is a fixed constant, \eqref{3.19} has a unique solution $\varphi=\varphi(r,\va)$
satisfying
$$
\|\varphi\|_*\leq C\Big(\frac 1\nu\Big)^{\frac m2+\epsilon},\ \ \ |c_l|\leq C\Big(\frac 1\nu\Big)^{\frac m2+\epsilon}.
$$
where $\epsilon>0$ is a small constant.
\begin{proof} First we recall that $\nu=k^{\frac{N-2s}{N-2s-m}}.$
Set
$$
\mathcal {N}=\left\{w:w\in
C^{\alpha}(\r^{N})\cap \mathscr{H},\|u\|_{*}\leq \frac{1}{\nu^{\frac{m}{2}}},\int_{\r^N} U^{2^*_s-2}_{x^i,\va}Z_{i,l}w=0\right\},
$$
where $0<\alpha<s$ and $i=1,2,...,k,l=1,2.$
Thus from Proposition\ref{prop3.2}, \eqref{3.19} is equivalent to
\begin{equation*}
\varphi=\mathcal
{A}(\varphi)=:L_{k}(N(\varphi))+L_{k}(l_{k}),
\end{equation*}
 $L_{k}$ is defined in Proposition \ref{prop3.2}.
We obtain
\begin{equation}\label{A1}
\begin{array}{ll}
\|\mathcal {A}(\varphi)\|_{*}&\leq
C\|N(\varphi)\|_{**}+C\|l_{k}\|_{**}\leq
C\|\varphi\|_{*}^{\min\{2^*_s-1,2\}}+C\Big(\ds\frac{1}{\nu}\Big)^{\frac{m}{2}+\epsilon}\vspace{0.2cm}\\
&\leq C\left(\ds\frac{1}{\nu}\right)^{\frac{m}{2}(\min\{2^*_s-1,2\})}+C\left(\ds\frac{1}{\nu}\right)^{\frac{m}{2}+\epsilon}\vspace{0.2cm}\\
&\leq C\left(\ds\frac{1}{\nu}\right)^{\frac{m}{2}+\epsilon}
 \leq
\ds\frac{1}{\nu^{\frac{m}{2}}}.
\end{array}
\end{equation}
Hence, $\mathcal {A}$ maps $\mathcal {N}$ to $\mathcal {N}.$
On the other hand,
$$
|N'(t)|\leq C|t|^{\min\{2^*_s-2,1\}}.
$$
We get
\begin{eqnarray*}
&&\|\mathcal {A}(\varphi_{1})-\mathcal
{A}(\varphi_{2})\|_{*}=\|L_{k}(N(\varphi_{1}))-L_{k}(N(\varphi_{2}))\|_{*}\leq C\|N(\varphi_{1})-N(\varphi_{2})\|_{**}\\
&\leq& C |N'(\varphi_{1}+\theta\varphi_{2})||\varphi_{1}-\varphi_{2}|\\
&\leq &C
(|\varphi_{1}|^{\min\{2^*_s-2,1\}}+|\varphi_{2}|^{\min\{2^*_s-2,1\}})
|\varphi_{1}-\varphi_{2}|\\
&\leq&C
(\|\varphi_{1}\|_{*}^{\min\{2^*_s-2,1\}}+\|\varphi_{2}\|_{*}^{\min\{2^*_s-2,1\}})
\|\varphi_{1}-\varphi_{2}\|_{*}
\Big(\sum_{i=1}^{k}\frac{1}{(1+|x-x^i|)^{\frac{N-2s}{2}+\tau}}\Big)^{\min\{2^*_s-1,2\}}\\
&\leq&\frac12\|\varphi_{1}-\varphi_{2}\|_{*}.
\end{eqnarray*}
Thus $\mathcal {A}$ is a contraction map.
Therefore, $\mathcal{A}$ is a contraction map from
$\mathcal {N}$ to $\mathcal {N}.$ Now applying the contraction mapping theorem, we can find a unique
$\varphi=\varphi(r,\va)\in \mathcal {N}$ such that
$$
\varphi=\mathcal {A}(\varphi).
$$
 Moreover, by Proposition \ref{prop3.2} we have
$$
\|\varphi\|_*\leq C\Big(\frac1\nu\Big)^{\frac m2+\epsilon}.
$$
Moreover, we get the estimate of $c_{l}$ from \eqref{3.11}. We also can see \eqref{3.16}.

\end{proof}
\end{proposition}

\section{Proof of the main result}
Let $F(r,\va)=I(U_{r,\va}+\varphi),$ where
$r=|x^1|$, $\varphi$ is the function obtained in Proposition
\ref{prop3.3}, and
$$
I(u)=\frac{1}{2}\int_{\r^N} |(-\Delta)^{\frac s2} u|^2-\frac{1}{2^*_s}\int_{\r^N} K\big(\frac{|x|}{\nu}\big)|u|^{2^*_s} .
$$
\begin{proposition}\label{prop4.1}
We have
\begin{eqnarray*}
F(r,\va)&=&I(U_{r,\va})+O\left(\frac{k}{\nu^{m+\epsilon}}\right)\\
&=&k\Big(A+\frac{B_0}{\va^{m}\nu^{m}}+\frac{B_1}{\va^{m-2}\nu^{m}}\Big(\nu r_{0}-|x^1|\Big)^{2}\\
&&\,\,\,\,\,\,-\sum_{i=2}^{k}\frac{B_2}{\va^{N-2s}|x^1-x^i|^{N-2s}}
+O\Big(\frac{1}{\nu^{m+\epsilon}}+\frac{1}{\nu^{m}}|\nu r_{0}-|x^1||^{3}\Big)\Big)\\
&=&k\Big(A+\frac{B_0}{\va^{m}\nu^{m}}+\frac{B_1}{\va^{m-2}\nu^{m}}\Big(\nu r_{0}-r\Big)^{2}\\
&&\,\,\,\,\,\,-\frac{B_3k^{N-2s}}{\va^{N-2s}r^{N-2s}}
+O\Big(\frac{1}{\nu^{m+\epsilon}}+\frac{1}{\nu^{m}}|\nu r_{0}-|x^1||^{3}+\frac{k}{r^{N-2s}}\Big)\Big).\\
\end{eqnarray*}
where $\epsilon>0$ is a fixed constant, $B_{i}>0,i=0,1,2,3$ are some constants.
\end{proposition}
\begin{proof}
Since
$$
\langle I'(U_{r,\va}+\varphi),\varphi\rangle=0,\,\,\,\forall \varphi\in \mathcal{N},
$$
there is $t\in(0,1)$ such that
\begin{eqnarray*}
F(r,\va)&=&I(U_{r,\va}+\varphi)\\
&=&I(U_{r,\va})-\frac{1}{2}D^{2}I(U_{r,\va}+t\varphi)(\varphi,\varphi)\\
&=&I(U_{r,\va})-\frac{1}{2}\int_{\r^N}\Big(\left|\nabla\varphi\right|^{2}
-(2^*_s-1)K\Big(\frac{|x|}{\nu}\Big)(U_{r,\va}+t\varphi)^{2^*_s-2})\varphi^{2}\Big)\\
&=&I(U_{r,\va})-\frac{1}{2}\int_{\r^N}(N(\varphi)+l_{k})\varphi\\
&&+\frac{2^*_s-1}{2}\int_{\r^N} K\big(\frac{|x|}{\nu}\big)\big((U_{r,\va}+t\varphi)^{2^*_s-2}-
U_{r,\va}^{2^*_s-2}\big)\varphi^{2}\\
&=&I(U_{r,\va})+O\Big(\int_{\r^N}\Big(|\varphi|^{2^*_s}+|N(\varphi)||\varphi|+|l_{k}||\varphi|\Big)\Big).
\end{eqnarray*}
However,
\begin{eqnarray*}
&&\int_{\r^N}(|N(\varphi)||\varphi|+|l_k||\varphi|)\\
&\leq &
C(\|N(\varphi)\|_{**}+\|l_{k}\|_{**})\|\varphi\|_{*}\int_{\r^N}\sum_{i=1}^{k}\frac{1}{(1+|x-x^{i}|)^{\frac{N+2s}{2}+\tau}}
\sum_{j=1}^{k}\frac{1}{(1+|x-x^j|)^{\frac{N-2s}{2}+\tau}}.
\end{eqnarray*}
Since $\tau=\frac{N-2s-m}{N-2s}$, $\sum\limits_{i=2}^k\frac{1}{|x^i-x_{1}|^{\tau}}\leq C,$ and Lemma \ref{lmA.1}, we obtain
\begin{eqnarray*}
&&\sum_{i=1}^k\frac{1}{(1+|x-x^i|)^{\frac{N+2s}{2}+\tau}}
\sum_{j=1}^{k}\frac{1}{(1+|x-x^j|)^{\frac{N-2s}{2}+\tau}}\\
&=&\sum_{i=1}^{k}\frac{1}{(1+|x-x^i|)^{N+2\tau}}
+\sum_{j=1}^{k}\sum_{j\neq i}
\frac{1}{(1+|x-x^i|)^{\frac{N+2s}{2}+\tau}}
\frac{1}{(1+|x-x^j|)^{\frac{N-2s}{2}+\tau}}\\
&\leq&\sum_{i=1}^k\frac{1}{(1+|x-x^i|)^{N+2\tau}}
+C\sum_{i=1}^{k}\frac{1}{(1+|x-x^i|)^{N+\tau}}\sum_{j=2}^{k}\frac{1}{|x^j-x^1|^{\tau}}\\
&\leq&C\sum_{i=1}^{k}\frac{1}{(1+|x-x^i|)^{N+\tau}}.
\end{eqnarray*}
Therefore, we see
\begin{eqnarray*}
\int_{\r^N}(|N(\varphi)||\varphi|+|l_k||\varphi|)&\leq&C(\|N(\varphi)\|_{**}
+\|l_{k}\|_{**})\|\varphi\|_{*}\int_{\r^N}\sum_{i=1}^{k}\frac{1}{(1+|x-x^i|)^{N+\tau}}\\
&\leq & Ck(\|N(\varphi)\|_{**}+\|l_{k}\|_{**})\|\varphi\|_{*}\leq
Ck\Big(\frac{1}{\nu^{m+\epsilon}}\Big).
\end{eqnarray*}
On the other hand, by H\"{o}lder inequality, we obtain
\begin{eqnarray*}
\int|\varphi|^{2^*_s}&\leq & C\|\varphi\|_{*}^{2^*_s}\int\Big(\sum_{i=1}^{k}\frac{1}{(1+|x-x^i|)^{\frac{N-2s}{2}+\tau}}\Big)^{2^*_s}\\
&\leq & C\|\varphi\|_{*}^{2^*_s}\int\Big(\sum_{i=1}^{k}\frac{1}{(1+|x-x^i|)^{N+\tau}}\Big)
\Big(\sum_{i=1}^{k}\frac{1}{(1+|x-x^i|)^{\tau}}\Big)^{2^*_s-1}\\
&\leq & C'\|\varphi\|_{*}^{2^*_s}\int\sum_{i=1}^{k}\frac{1}{(1+|x-x^i|)^{N+\tau}}\\
&\leq& C k\|\varphi\|_{*}^{2^*_s}\\
&\leq& Ck\Big(\frac{1}{\nu^{m+\epsilon}}\Big).
\end{eqnarray*}
Therefore applying Proposition \ref{propA.1} we have
\begin{eqnarray*}
F(r,\va)
&=&I(U_{r,\va})+O\left(\frac{k}{\nu^{m+\epsilon}}\right)\\
&=&k\Big(A+\frac{B_{0}}{\va^{m}\nu^{m}}+\frac{B_{1}}{\va^{m-2}\nu^{m}}\Big(\mu r_{0}-|x^1|\Big)^{2}\\
&&\,\,\,\,\,\,-\sum_{i=2}^{k}\frac{B_{2}}{\va^{N-2s}|x^1-x^i|^{N-2s}}
+O\Big(\frac{1}{\nu^{m+\epsilon}}+\frac{1}{\nu^{m}}|\nu r_{0}-|x^1||^{3}\Big)\Big).
\end{eqnarray*}
There is a constant $B_3>0,$ such that
$$
\sum_{i=2}^k\frac {1}{|x^i-x^1|^{N-2s}}=\frac{B_3 k^{N-2s}}{|x^1|^{N-2s}}+O\Big(\frac k{|x^1|^{N-2s}}\Big).
$$
So
\begin{eqnarray*}
F(r,\va)&=&k\Big(A+\frac{B_{0}}{\va^{m}\nu^{m}}+\frac{B_{1}}{\va^{m-2}\nu^{m}}\Big(\nu r_{0}-r\Big)^{2}\\
&&\,\,\,\,\,\,-\frac{B_{3}k^{N-2s}}{\va^{N-2s}r^{N-2s}}
+O\Big(\frac{1}{\nu^{m+\epsilon}}+\frac{1}{\nu^{m}}|\nu r_{0}-|x^1||^{3}+\frac{k}{r^{N-2s}}\Big)\Big).
\end{eqnarray*}
\end{proof}

\begin{proposition}\label{prop4.2}
\begin{eqnarray*}
\frac{\partial F(r,\va)}{\partial \va}&=&k\Big(-\frac{B_0m}{\va^{m+1}\nu^{m}}+\sum_{i=2}^{k}\frac{B_2(N-2s)}{\va^{N-2s+1}|x^1-x^i|^{N-2s}}\\
&&\ \ \ \ +O\Big(\frac{1}{\nu^{m+\epsilon}}+\frac{1}{\nu^{m}}|\nu r_{0}-|x^1||^{2}\Big)\Big)\\
&=&k\Big(-\frac{B_0m}{\va^{m+1}\nu^{m}}+\frac{B_3(N-2s)k^{N-2s}}{\va^{N-2s+1}r^{N-2s}}\\
&&\ \ \ \ +O\Big(\frac{1}{\nu^{m+\epsilon}}+\frac{1}{\nu^{m}}|\nu r_{0}-|x^1||^{2}+\frac{k}{r^{N-2s}}\Big)\Big).\\
\end{eqnarray*}
where $\epsilon>0$ is a fixed constant.
\end{proposition}
\begin{proof}
We have
\begin{equation*}\label{4.1}
\begin{array}{ll}
\ds\frac{\partial F(r,\va)}{\partial\va}
&=\Big\langle I'(U_{r,\va}+\varphi),\ds\frac{\partial U_{r,\va}}{\partial \va}+\ds\frac{\partial \varphi}{\partial \va}\Big\rangle\vspace{0.2cm}\\
&=\Big\langle I'(U_{r,\va}+\varphi),\ds\frac{\partial
U_{r,\va}}{\partial \va}\Big\rangle
+\ds\sum_{l=1}^{2}\ds\sum_{i=1}^{k}c_{l}\Big\langle U_{x^i,\va}^{2^*_s-2}
Z_{i,l},\frac{\partial \varphi}{\partial
\va}\Big\rangle\vspace{0.2cm}\\
&=\ds\frac{\partial I(U_{r,\va})}{\partial \va}-
\int_{\r^N} K\big(\frac{|x|}{\nu}\big)\big[(U_{r,\va}+\varphi)^{2^*_s-1}-U_{r,\va}^{2^*_s-1}\big]\ds\frac{\partial
U_{r,\va}}{\partial \va}\vspace{0.2cm}\\
&\,\,\,\,\,\,\,+\ds\sum_{l=1}^{2}\ds\sum_{i=1}^{k}c_{l}\Big\langle U_{x^i,\va}^{2^*_s-2}
Z_{i,l},\frac{\partial \varphi}{\partial
\va}\Big\rangle.
\end{array}
\end{equation*}
Note that
$$
\Big\langle U_{x^i,\va}^{2^*_s-2}Z_{i,l},\frac{\partial \varphi}{\partial
\va}\Big\rangle =-\Big\langle \frac{\partial
\big(U_{x^i,\va}^{2^*_s-2}Z_{i,l}\big)}{\partial \va},\varphi\Big\rangle.
$$
Thus,
\begin{equation*}\label{4.2}
\begin{array}{ll}
&\Big|\ds\sum_{i=1}^{k}c_{l}\Big\langle
U_{x^i,\va}^{2^*_s-2}Z_{i,l},\frac{\partial \varphi}{\partial
\va}\Big\rangle\Big|\vspace{0.2cm}\\
\leq &
C |c_{l}|\|\varphi\|_{*}\ds\int_{\r^N}\ds\sum_{i=1}^{k}\frac{1}{(1+|x-x^i|)^{\frac{N-2s}{2}+\tau}}
\ds\sum_{j=1}^{k}\frac{1}{(1+|x-x^{j}|)^{N+2s}}\vspace{0.2cm}\\
=&C|c_{l}|\|\varphi\|_{*}\ds\sum_{i=1}^{k}\ds\Big(\int_{\r^N}\frac{1}{(1+|x-x^i|)^{\frac{3N+2s}{2}+\tau}}
+\int_{\r^N}\frac{1}{(1+|x-x^i)^{\frac{N-2s}{2}+\tau}}\sum_{j\neq
i}\frac{1}{(1+|x-x^{j}|)^{N+2s}}\Big)\vspace{0.2cm}\\
\leq &C|c_{l}|\|\varphi\|_{*}\ds\sum_{i=1}^{k}\ds\int_{\r^N} \Big(\frac{1}{(1+|x-x^i|)^{\frac{3N+2s}{2}+\tau}}+
\frac{1}{(1+|x-x^i|)^{N+\frac{m}{2}}}\sum_{i=2}^{k}\frac{1}{|x^i-x^1|^{\frac{N+2s-m}{2}+\tau }}\Big)\vspace{0.2cm}\\
 \leq&
C\ds\frac{k}{\nu^{m+\epsilon}}.
\end{array}
\end{equation*}
On the other hand, from $\varphi\in\mathcal{N}$ we have
\begin{eqnarray*}
&& \ds\int_{\r^N}K\big(\frac{|x|}{\nu}\big)[(U_{r,\va}+\varphi)^{2^*_s-1}-U_{r,\va}^{2^*_s-1}]\frac{\partial U_{r,\va}}{\partial \va}\vspace{0.2cm}\\
&=&\ds\int_{\r^N}(2^*_s-1)K\big(\frac{|x|}{\nu}\big)U_{r,\va}^{2^*_s-1}\frac{\partial U_{r,\va}}{\partial \va}\varphi+O\Big(\int_{\r^N}|\varphi|^2\Big)\vspace{0.2cm}\\
&=&\ds\int_{\r^N}(2^*_s-1)K\big(\frac{|x|}{\nu}\big)\Big(U_{r,\va}^{2^*_s-1}\frac{\partial U_{r,\va}}{\partial \va}\varphi-\sum\limits_{i=1}^k U_{x^i,\va}\frac{\partial U_{x^i,\va}}{\partial \va}\Big)\varphi\vspace{0.2cm}\\
&&+(2^*_s-1)\ds\sum\limits_{i=1}^k\ds\int_{\r^N}\Big(K\big(\frac{|x|}{\nu}\big)-1\Big)U_{x^i,\va}^{2^*_s-2} \frac{\partial U_{x^i,\va}}{\partial \va}\varphi +O\Big(\int_{\r^N}|\varphi|^2\Big)\vspace{0.2cm}\\
&=&k\ds\int_{\Omega_1}(2^*_s-1)K\big(\frac{|x|}{\nu}\big)\Big(U_{r,\va}^{2^*_s-1}\frac{\partial U_{r,\va}}{\partial \va}\varphi-\sum\limits_{i=1}^k U_{x^i,\va}\frac{\partial U_{x^i,\va}}{\partial \va}\Big)\varphi\vspace{0.2cm}\\
&&+k\ds\int_{\r^N}(2^*_s-1)\Big(K\big(\frac{|x|}{\nu}\big)-1\Big)U_{x^1,\va}^{2^*_s-2} \frac{\partial U_{x^1,\va}}{\partial \va}\varphi +O\Big(\int_{\r^N}|\varphi|^2\Big)\vspace{0.2cm}\\
&\leq &Ck \ds\int_{\Omega_1}(2^*_s-1)\Big(U_{x^1,\va}^{2^*_s-2}\sum_{i=2}^kU_{x^i,\va}+\sum_{i=2}^kU_{x^i,\va}^{2^*_s-1}\Big)\varphi\vspace{0.2cm}\\
&&+\Big(\ds\int_{\big||x|-\nu r_0\big|\leq\sqrt{\nu}}+\ds\int_{\big||x|-\nu r_0\big|\geq\sqrt{\nu}}\Big(K\big(\frac{|x|}{\nu}\big)-1\Big)U_{x^1,\va}^{2^*_s-2} \frac{\partial U_{x^1,\va}}{\partial \va}\varphi\Big) +O\Big(\int_{\r^N}|\varphi|^2\Big)\vspace{0.2cm}\\
 &\leq&
C\ds\frac{k}{\nu^{m+\epsilon}}.
\end{eqnarray*}
This completes our proof.
\end{proof}

Let
$\va_{0}$ be the solution of
$$
-\frac{B_0m}{\va^{m+1}}+\frac{B_3(N-2s)}{\va^{N-2s+1}r_{0}^{N-2s}}=0.
$$
Then
$$
\va_{0}=\Big(\frac{B_3(N-2s)}{B_0mr_{0}^{N-2s}}\Big)^{\frac{1}{N-2s-m}}.
$$
Define
$$
\Omega=\Big\{(r,\va):r\in\Big[\nu r_{0}-\frac{1}{\nu^{\bar{\theta}}},\nu r_{0}+\frac{1}{\nu^{\bar{\theta}}}\Big],
\va\in\Big[\va_{0}-\frac{1}{\nu^{\frac{3}{2}\bar{\theta}}},\va_{0}+\frac{1}{\nu^{\frac{3}{2}\bar{\theta}}}\Big]\Big\},
$$
where $\bar{\theta}>0$ is a small constant.

For any $(r,\va)\in \Omega,$ we have
$$
\frac{r}{\nu}=r_{0}+O\Big(\frac{1}{\nu^{1+\bar{\theta}}}\Big).
$$
Thus,
$$
r^{N-2s}=\nu^{N-2s}\Big(r_{0}^{N-2s}+O\Big(\frac{1}{\nu^{1+\bar{\theta}}}\Big)\Big).
$$
So
\begin{equation}\label{F}
\begin{array}{ll}
F(r,\va)
&=k\Big(A+\Big(\ds\frac{B_0}{\va^{m}}-\frac{B_3}{\va^{N-2s}r_{0}^{N-2s}}\Big)\ds\frac{1}{\nu^{m}}
+\ds\frac{B_1}{\va^{m-2}\nu^{m}}(\nu r_{0}-r)^{2}\vspace{0.2cm}
\\
&\,\,\,\,\,\,\,\,+O\Big(\ds\frac{1}{\nu^{m+\epsilon}}+\ds\frac{1}{\nu^{m}}|\nu r_{0}-r|^{3}+\frac{k}{\nu^{N-2s}}\Big)\Big),
\,\,\,(r,\va)\in \Omega,
\end{array}
\end{equation}
and
\begin{equation}\label{F'}
\begin{array}{ll}
\ds\frac{\partial F(r,\va)}{\partial\va}
&=k\Big(\Big(-\ds\frac{B_0m}{\va^{m+1}}+\ds\frac{B_3(N-2s)}{\va^{N-2s+1}r_{0}^{N-2s}}\Big)\ds\frac{1}{\nu^{m}}\\
&\,\,\,\,\,\,\,\,+O\Big(\ds\frac{1}{\nu^{m+\epsilon}}+\ds\frac{1}{\nu^{m}}|\nu r_{0}-r|^{2}+\frac{k}{\nu^{N-2s}}\Big)\Big)\,\,\,(r,\va)\in \Omega.
\end{array}
\end{equation}
Now, we define
$$
\tilde{F}(r,\va)=-F(r,\va),\,\,\,\,(r,\va)\in \Omega.
$$
Let
$$
\alpha_{1}=k\Big(-A-\Big(\frac{B_0}{\va_{0}^{m}}
-\frac{B_3}{\va_{0}^{N-2s}r_{0}^{N-2s}}\Big)\frac{1}{\nu^{m}}
-\frac{1}{\nu^{m+\frac{5}{2}\bar{\theta}}}\Big),\,\,\,
\alpha_{2}=k(-A+\eta),
$$
where $\eta>0$ is a small constant.
For $c\in\r$, define
$$
\tilde{F}^{c}=\left\{(r,\va)\in
\Omega,\,\,\,\tilde{F}(r,\va)\leq c\right\}.
$$
Consider
$$
\left\{%
\begin{array}{ll}
    \ds\frac{dr}{dt}=-D_{r}\tilde{F},\,\,& t>0;\vspace{0.2cm} \\
    \ds\frac{d \va}{dt}=-D_{\va}\tilde{F},\,\,\,&t>0;\vspace{0.2cm} \\
    (r,\va)\in \tilde{F}^{\alpha_{2}}.
\end{array}%
\right.
$$
Then we have
\begin{lemma}\label{lm4.3}
The flow $(r(t),\va(t))$ does not leave
 $\Omega$ before it reaches $ \tilde{F}^{\alpha_{1}}.$
\end{lemma}
\begin{proof}
If $\va=\va_{0}+\frac{1}{\nu^{\frac{3}{2}\bar{\theta}}}$, observing that
$|r-\nu r_{0}|\leq\frac{1}{\nu^{\bar{\theta}}}$, it follows from \eqref{F'} that
$$
\begin{array}{ll}
\ds\frac{\partial \tilde{F}(r,\va)}{\partial\va}
=k\Big(c'\frac{1}{\nu^{m+\frac{3}{2}\bar{\theta}}}+O\Big(\frac{1}{\nu^{m+2\bar{\theta}}}\Big)\Big)
>0.
\end{array}
$$
Hence the flow does not leave $\Omega.$

Similarly, if $\va=\va_{0}-\frac{1}{\nu^{\frac{3}{2}\bar{\theta}}}$, it follows from \eqref{F'} that
$$
\ds\frac{\partial \tilde{F}(r,\va)}{\partial\va}
=k\Big(-c'\frac{1}{\nu^{m+\frac{3}{2}\bar{\theta}}}+O\Big(\frac{1}{\nu^{m+2\bar{\theta}}}\Big)\Big)
<0.
$$
Therefore the flow does not leave $\Omega.$

Now suppose that $|r-\nu r_{0}|=\frac{1}{\nu^{\bar{\theta}}}$.
Since $|\va-\va_{0}|\leq\frac{1}{\nu^{\frac{3}{2}\bar{\theta}}}$, we see
\begin{eqnarray*}
\frac{B_0}{\va^{m}}
-\frac{B_3}{\va^{N-2s}r_{0}^{N-2s}}&=&\frac{B_0}{\va_{0}^{m}}
-\frac{B_3}{\va_{0}^{N-2s}r_{0}^{N-2s}}+O(|\va-\va_{0}|^{2})\\
&=&\frac{B_{1}}{\va_{0}^{m}}
-\frac{B_{4}}{\va_{0}^{N-2s}r_{0}^{N-2s}}+O\left(\frac{1}{\nu^{3\bar{\theta}}}\right).
\end{eqnarray*}

So it follows from \eqref{F} that
\begin{equation}\label{4.6}
\begin{array}{ll}
\tilde{F}(r,\Lambda)&=
k\Big(-A-\Big(\ds\frac{B_{1}}{\va_{0}^{m}}-\frac{B_{4}}{\va_{0}^{N-2s}r_{0}^{N-2s}}\Big)\ds\frac{1}{\nu^{m}}
-\ds\frac{B_{2}}{\va_{0}^{m-2}\nu^{m}}(\nu r_{0}-r)^{2}
+O\Big(\ds\frac{1}{\nu^{m+3\bar{\theta}}}\Big)\Big)\vspace{0.2cm}
\\
&\leq
k\Big(-A-\Big(\ds\frac{B_{1}}{\va_{0}^{m}}-\frac{B_{4}}{\va_{0}^{N-2s}r_{0}^{N-2s}}\Big)\ds\frac{1}{\nu^{m}}
-\ds\frac{B_{2}}{\va_{0}^{m-2}\nu^{m+2\bar{\theta}}}
+O\Big(\ds\ds\frac{1}{\nu^{m+3\bar{\theta}}}\Big)\Big)\vspace{0.2cm}\\
&<\alpha_{1}.
\end{array}
\end{equation}
\end{proof}

\begin{proof}[\textbf{Proof of Theorem \ref{th3}}]
 We will prove that
$\tilde{F},$ and therefore $F$ has a critical point in $\Omega.$

Define
$$
\begin{array}{ll}
\Gamma=\Big\{&g:g(r,\va)=(g_{1}(r,\va),
g_{2}(r,\va))\in
\Omega,(r,\va)\in\Omega,
\vspace{0.2cm}\\
&g(r,\va)=(r,\va),\text{if}\,\,|r-\nu r_{0}|=\ds\frac{1}{\nu^{\bar{\theta}}}\Big\}.
\end{array}
$$
Denote
$$
c=\ds\inf_{g\in\Gamma}\max_{(r,\va)\in\Omega}\tilde{F}
(g(r,\va)).
$$
We claim that $c$ is a critical value of $\tilde{F}$. In order
to prove this, we have to prove
$$
\begin{array}{ll}
& (i) \,\,\,\,\alpha_{1}<c<\alpha_{2};
\vspace{0.2cm}\\
& (ii)\,\,\ds\sup_{|r-\mu r_{0}|=\frac{1}{\nu^{\bar{\theta}}}}\tilde{F}
(g(r,\va))<\alpha_{1},\,\,\,\forall g\in\Gamma.
\end{array}
$$
In order to prove (ii), let $g \in\Gamma.$ Then for any $r$
with $|r-\nu r_{0}|=\frac{1}{\nu^{\bar{\theta}}}$, we have
$g(r,\va)=(r,\va).$ Hence, from
\eqref{4.6}, we obtain
$$
\tilde{F}\left(g(r,\va)\right)=\tilde{F}(r,\va)
<\alpha_{1}.
$$
Now we prove (i). It is obvious that $c<\alpha_{2}.$
For any $\gamma=(\gamma_{1},\gamma_{2})\in\Gamma$, then
$g_{1}(r,\va)=r$, if
$|r-\nu r_{0}|=\frac{1}{\nu^{\bar{\theta}}}$. Define
$$
\tilde{g}_{1}(r)=g_{1}\left(r,\va_{0}\right).
$$
Then $\tilde{g}_{1}(r)=r$ if $|r-\nu r_{0}|=\frac{1}{\nu^{\bar{\theta}}}$.
Hence there is an
$\tilde{r}\in\left(\nu r_{0}-\frac{1}{\nu^{\bar{\theta}}},\nu r_{0}+\frac{1}{\nu^{\bar{\theta}}}\right)$
such that
$$
\tilde{g}_{1}(\tilde{r})=\nu r_{0}.
$$
Let
$\tilde{\va}=g_{2}\left(\tilde{r},\va_{0}\right).$
Then it follows from \eqref{F} that
$$
\begin{array}{ll}
&\ds\max_{(r,\va)\in\Omega}\tilde{F}(g(r,\va))\geq
\tilde{F}\Big(g\Big(\tilde{r},\va_{0}\Big)\Big)
=\tilde{F}(\nu r_{0},\tilde{\va})\vspace{0.2cm}\\
=&k\Big(-A-\Big(\ds\frac{B_0}{\tilde{\va}^{m}}
-\ds\frac{B_3}{\tilde{\va}^{N-2s}r_0^{N-2s}}\Big)\ds\frac{1}{\nu^{m}}
+O\Big(\ds\frac{1}{\nu^{m+\epsilon}}+\frac{k}{\nu^{N-2s}}\Big)\Big)
\vspace{0.2cm}\\
\ge &k\Big(-A-\Big(\ds\frac{B_0}{\va_0^{m}}
-\ds\frac{B_3}{\va_{0}^{N-2s}r_0^{N-2s}}\Big)\ds\frac{1}{\nu^{m}}
+O\Big(\ds\frac{1}{\nu^{m+3\bar{\theta}}}\Big)\Big)>\alpha_{1}.
\end{array}
$$
\end{proof}

Now we come to give the sketch of proof for Theorem \ref{th4}.
Recall that
$$\bar U_{r,\va}(x)=\sum_{i=1}^{2k}(-1)^{i-1}U_{\bar x^i}(x).$$
We will seek for  a
 solution for  equation \eqref{eq} of the form $\bar U_{r,\va}(x)+\bar\varphi$
 with $\bar\varphi=\bar\varphi(r,\varepsilon)$ solved \eqref{3.19}.
To this end,
we should also perform the same procedure as the proof of Theorem
\ref{th3}.

Proceeding as we proved Propositions \ref{prop4.1} and \ref{prop4.2}, we conclude that
\begin{eqnarray*}
\bar F(r,\va)&:=&I(\bar U_r(x)+\bar\varphi)=k\Big(A-\frac{B'_{0}}{\va^{m}\nu^{m}}-\frac{B'_{1}}{\va^{m-2}\nu^{m}}\Big(\nu r_{0}-r\Big)^{2}\\
&&\,\,\,\,\,\,+\frac{B'_{3}k^{N-2s}}{\va^{N-2s}r^{N-2s}}
+O\Big(\frac{1}{\nu^{m+\epsilon}}+\frac{1}{\nu^{m}}|\nu r_{0}-|\bar x^1||^{3}+\frac{k}{r^{N-2s}}\Big)\Big)
\end{eqnarray*}
and
\begin{eqnarray*}
\frac{\partial \bar F(r,\va)}{\partial \va}
&=&k\Big(\frac{B'_0m}{\va^{m+1}\nu^{m}}-\frac{B'_3(N-2s)k^{N-2s}}{\va^{N-2s+1}r^{N-2s}}\\
&&\ \ \ \ +O\Big(\frac{1}{\nu^{m+\epsilon}}+\frac{1}{\nu^{m}}|\nu r_{0}-|\bar x^1||^{2}+\frac{k}{r^{N-2s}}\Big)\Big).\\
\end{eqnarray*}

Let $\varepsilon_0$ and $\Omega$ be given as above. Define
$$
\alpha_{1}=k\Big(A-\Big(\frac{B'_0}{\va_{0}^{m}}
-\frac{B'_3}{\va_{0}^{N-2s}r_{0}^{N-2s}}\Big)\frac{1}{\nu^{m}}
+\frac{1}{\nu^{m+\frac{5}{2}\bar{\theta}}}\Big),\,\,\,
\alpha_{2}=k(A+\eta),
$$
where $\eta>0$ is a small constant so that $\alpha_{1}<\alpha_{2}$.

Then proceeding as done in the proof of Theorem
\ref{th3}, we can prove Theorem \ref{th4}.

\appendix

\section{{ Energy expansion}}

In this section, we will give some basic estimates and the energy expansion for the
approximate solutions. Recall
$$
x^i=\Big(r\cos\frac{2(i-1)\pi}{k},r\sin\frac{2(i-1)\pi}{k}, 0\Big),
\ i=1,\cdots,k,\ 0\in \r^{N-2},
$$
$$
\Omega_i=\Big\{x=(x',x'')\in \r^2\times
\r^{N-2}:\langle\frac{x'}{|x'|},\frac{x^i}{|x^i|}\rangle\geq
\cos\frac{\pi}{k} \Big\},i=1,2,\cdots,k,
$$

$$
U_{r,\va}(x)=C_{N,s}\sik\frac{\va^{\frac{N-2s}{2}}}{(1+\va^2|x-x^i|)^\frac{N-2s}{2}},
$$
and

$$
I(u)=\frac{1}{2}\int_{\r^{2N}}\frac{|u(x)-u(y)|^2}{|x-y|^{N+2s}}-\frac{1}{2^*_s}\int_{\r^N}K\Big(\frac{|x|}{\nu}\Big)|u|^{2^*_s}
=\frac{1}{2}\int_{\r^{N}}|(-\Delta)^{\frac s2}u|^2-\frac{1}{2^*_s}\int_{\r^N}K\Big(\frac{|x|}{\nu}\Big)|u|^{2^*_s}.
$$

Similar Lemma B.1 and B.2 in \cite{wy1}, we have,

\begin{lemma}\label{lmA.1} For any constant $0<\sigma\leq
\min\{\al,\be\}$, there is a constant $C>0$, such that
$$
\frac{1}{(1+|y-x^i|)^{\al}}\frac{1}{(1+|y-x^j|)^{\be}}\leq
\frac{C}{|x^i-x^j|^{\sigma}}\Big(\frac{1}{(1+|y-x^i|)^{\al+\be-\sigma}}+\frac{1}{(1+|y-x^j|)^{\al+\be-\sigma}}\Big).
$$
\end{lemma}

\begin{lemma}\label{lmA.2} For any constant $0<\kappa<
N-2s$, there is a constant $C>0$, such that
$$
\int_{\r^N}\frac{1}{|x|^{N-2s}}\frac{1}{(1+|y-x|)^{2s+\kappa}}dx \leq \frac{C}{(1+|y|)^\kappa}.
$$
\end{lemma}

\begin{lemma}\label{lmA.3}There is a small $\alpha>0$, such that
$$
\int_{\r^N}\frac{1}{|x-y|^{N-2s}}U_{r,\va}^{\frac{4s}{N-2s}}(y)\sik \frac{1}{(1+|y-x^i|)^{\frac{N-2s}{2}+\sigma}}dy\leq C\sik\frac{1}{(1+|x-x^i|)^{\frac{N-2s}{2}+\sigma+\alpha}}.
$$
\end{lemma}
\begin{proof} Since
$$
|x^i-x^1|=2|x^1|\sin\frac{(i-1)\pi}{k},\ \ i=2,\cdots,k,
$$
we have
\begin{eqnarray*}
\sum_{i=2}^k\frac{1}{|x^i-x^1|^\eta}&=&\frac{1}{(2|x^1|)^\eta}\sum_{i=2}^k\frac{1}{(\sin\frac{(i-1)\pi}{k})^\eta}\\
& =& \left\{\begin{array}{ll}
        \frac{2}{(2|x^1|)^\eta}\sum_{i=2}^{\frac{k}{2}}\frac{1}{(\sin\frac{(i-1)\pi}{k})^\eta}+\frac{1}{(2|x^1|)^\eta}, & \text{if}\  k\ \text{is \ even},\vspace{2mm}\\
       \frac{2}{(2|x^1|)^\eta}\sum_{i=2}^{[\frac{k}{2}]}\frac{1}{(\sin\frac{(i-1)\pi}{k})^\eta}, & \text{if}\ k\ \text{is \ odd}.
    \end{array}\right.
\end{eqnarray*}
But
$$
0<c'\leq\frac{\sin\frac{(i-1)\pi}{k}}{\frac{(i-1)\pi}{k}}\leq C'',\
\ i=2,\cdots,[\frac{k}{2}].
$$
So, for $\nu=k^{\frac{N-2s}{N-2s-m}}$ and any $\eta\geq \frac{N-2s-m}{N-2s}$ ,
\begin{eqnarray*}
\sum_{i=2}^k\frac{1}{|x^i-x^1|^{\eta}}\leq \frac{C k^\eta}{\nu^\eta}\sum_{i=2}^k\frac{1}{i^\eta}
=\left\{\begin{array}{ll}
       \frac{C k^\eta\ln k}{\nu^\eta}\leq C, & \eta\geq 1,\vspace{2mm}\\
      \frac{C k}{\nu^\eta}, & \eta<1.
    \end{array}\right.
\end{eqnarray*}
For any $x\in \Omega_1$, we have for $i\neq 1$,
$
 |x-x^i|\geq|x-x^1|.
$
By using Lemma \ref{lmA.1}, we obtain
\begin{equation*} \begin{split}
\sum_{i=2}^k\frac{1}{(1+|x-x^i|)^{N-2s}}&\leq \sum_{i=2}^k C\frac{1}{(1+|x-x^1|)^{N-2s-\eta}}\frac{1}{(1+|x-x^i|)^{\eta}}\\
&\leq
\frac{1}{(1+|x-x^1|)^{N-2s-\eta}}\sum_{i=2}^k\frac{1}{|x^1-x^i|^\eta}\\
&\leq \frac{C}{(1+|x-x^1|)^{N-2s-\eta}}.
\end{split}
\end{equation*}
Thus,
$$
U_{r,\va}^{\frac{4s}{N-2s}}\leq \frac{C}{(1+|x-x^1|)^{4s-\frac{4s\eta}{N-2s}}}.
$$
It follows from Lemmas \ref{lmA.1} and \ref{lmA.2} that
\begin{eqnarray*}
&&\int_{\Omega_1}\frac{1}{|y-x|^{N-2s}}U_{r,\va}^{\frac{4s}{N-2s}}\sik\frac{1}{(1+|x-x^i|)^{\frac{N-2s}{2}+\sigma}}\\
&\leq&C \int_{\Omega_1}\frac{1}{|y-x|^{N-2s}}\frac{1}{(1+|x-x^1|)^{4s-\frac{4s\eta}{N-2s}}}\sik\frac{1}{(1+|x-x^i|)^{\frac{N-2s}{2}+\sigma}}\\
&\leq&C \int_{\Omega_1}\frac{1}{|y-x|^{N-2s}}\Big(\frac{1}{(1+|x-x^1|)^{\frac{N+6s}{2}+\sigma-\frac{4s\eta}{N-2s}}}+\frac{1}{(1+|x-x^1|)^{\frac{N+6s}{2}
+\sigma-\frac{(N+2s)\eta}{N-2s}}}\ds\sum_{i=2}^{k}\frac{1}{|x^i-x^1|^\eta}\Big)\\
&\leq&C \int_{\Omega_1}\frac{1}{|y-x|^{N-2s}}\frac{1}{(1+|x-x^1|)^{\frac{N+6s}{2}+\sigma-\frac{(N+2s)\eta}{N-2s}}}\\
&\leq&C \frac{1}{(1+|y-x^1|)^{\frac{N+2s}{2}+\sigma-\frac{(N+2s)\eta}{N-2s}}}=C \frac{1}{(1+|y-x^1|)^{\frac{N-2s}{2}+\sigma+2s-\frac{(N+2s)\eta}{N-2s}}}\\
&\leq&C \frac{1}{(1+|y-x^1|)^{\frac{N-2s}{2}+\sigma+\alpha}},
\end{eqnarray*}
since $2s-\frac{(N+2s)\eta}{N-2s}>0$ when $m>N-2s(1+\frac{(N-2s)^2}{N+2s})$.

Thus, we obtain
\begin{eqnarray*}
&&\int_{\r^N}\frac{1}{|y-x|^{N-2s}}U_{r,\va}^{\frac{4s}{N-2s}}\sik\frac{1}{(1+|x-x^i|)^{\frac{N-2s}{2}+\sigma}}\\
&=&\sum_{j=1}^k \int_{\Omega_j}\frac{1}{|y-x|^{N-2s}}U_{r,\va}^{\frac{4s}{N-2s}}\sik\frac{1}{(1+|x-x^i|)^{\frac{N-2s}{2}+\sigma}}\\
&\leq&C \sik\frac{1}{(1+|y-x^i|)^{\frac{N-2s}{2}+\sigma+\alpha}}.
\end{eqnarray*}
\end{proof}

\begin{proposition}\label{propA.1} We have
\begin{equation*}\begin{split}
I(U_{r,\va})=&k\Big(A+\frac{B_0}{\va^{m}\nu^m}+\frac{B_1}{\va^{m-2}\nu^m}(\nu r_0-r)^2-\sum_{i=2}^k\frac{B_2}{\va^{N-2s}|x^1-x^i|^{N-2s}}\\
&+O\big(\frac{1}{\nu^{m+\sigma}}+\frac{1}{\nu^{m}}\big|\nu r_0-|x^1|\big|^{2+\sigma}\big)\Big)
\end{split}
\end{equation*}
 where
$B_0\in [C_1,C_2],\ 0<C_1<C_2,$ $A=\frac{s}{N}\int_{\r^N}U_{0,1}^{2^*_s},\  B_i,\ C_i,\ i=1,2,$ are some constants, and $r=|x^1|$.
\end{proposition}
\begin{proof}
 Using the symmetry, we have

\begin{equation}\label{A.1} \begin{split}
\langle U_r,U_r\rangle_s&=\sik\sum_{j=1}^{k}\int_{\r^N}U_{x^j,\va}^{2^*_s-1}U_{x^i,\va}\\
&=k\Bigl(\int_{\r^N}U_{0,1}^{2^*_s}+\sum_{i=2}^{k}\int_{\r^N}U_{x^1,\va}^{2^*_s-1}U_{x^i,\va}\Bigr)\\
 &=k \int_{\r^N}U_{0,1}^{2^*_s}+k \sum_{i=2}^{k}\int_{\r^N}U_{x^1,\va}^{2^*_s-1}U_{x^i,\va}.\\
\end{split}
\end{equation}
It follows from Lemma \ref{lmA.1} that
\begin{equation*}\label{A.2} \begin{split}
&\sum_{i=2}^k\int_{\r^N}U_{x^1,\va}^{2^*_s-1}U_{x^i,\va}=C\int_{\r^N}\Big(\frac{1}{(1+\va|x-x^1|)^{N-2s}}\Big)^{2^*_s-1}\sum_{i=2}^k\frac{1}{(1+\va|x-x^i|)^{N-2s}}\\
&\le C
\sum_{i=2}^k\frac{1}{\va^{N-2s}|x^1-x^i|^{N-2s}}\int_{\r^N}\frac{1}{(1+\va|x-x^1|)^{(N-2s)(2^*_s-1)}}+O\Big(\sum_{i=2}^k\frac{1}{\va^{N-2s+\sigma}|x^1-x^i|^{N-2s+\sigma}}\Big)\\
&=\sum_{i=2}^k\frac{C_1}{\va^{N-2s}|x^1-x^i|^{N-2s}}+O\Big(\sum_{i=2}^k\frac{1}{|x^1-x^i|^{N-2s+\sigma}}\Big).
\end{split}
\end{equation*}
However,
\begin{eqnarray*}
\sum_{i=2}^k\int_{\r^N}U_{x^1,\va}^{2^*_s-1}U_{x^i,\va}&=&\int_{\r^N}\Big(\frac{1}{(1+\va|x-x^1|)^{N-2s}}\Big)^{2^*_s-1}\sum_{i=2}^k\frac{1}{(1+\va|x-x^i|)^{N-2s}}\\
&\ge&
\sum_{i=2}^k\int_{B_{\frac{|x^1-x^i|}{2}}(x^1)}\Big(\frac{1}{(1+\va|x-x^1|)^{N-2s}}\Big)^{2^*_s-1}\frac{1}{(1+\va|x-x^i|)^{N-2s}}\\
&&+\sum_{i=2}^k
\int_{B_{\frac{|x^1-x^i|}{2}}(x^i)}\Big(\frac{1}{(1+\va|x-x^1|)^{N-2s}}\Big)^{2^*_s-1}\frac{1}{(1+\va|x-x^i|)^{N-2s}}\\
&\ge&\sum_{i=2}^k\frac{C_2}{\va^{N-2s}|x^1-x^i|^{N-2s}}+O\Big(\sum_{i=2}^k\frac{1}{|x^1-x^i|^{N-2s+\sigma}}\Big).
\end{eqnarray*}

Hence, there exists $B_0$ in $[C_2,\,\,C_1]$, where $C_1$ and $C_2$ are independent of $k$, such that
\begin{equation}\label{A.2} \begin{split}
&\sum_{i=2}^k\int_{\r^N}U_{x^1,\va}^{2^*_s-1}U_{x^i,\va}=\sum_{i=2}^k\frac{B_0}{\va^{N-2s}|x^1-x^i|^{N-2s}}+O\Big(\sum_{i=2}^k\frac{1}{|x^1-x^i|^{N-2s+\sigma}}\Big).
\end{split}
\end{equation}
Now, by symmetry, we see
\begin{eqnarray*}
\int_{\r^N}K(\frac{|x|}{\nu})U_r^{2^*_s}&=&k\int_{\Omega_1}K(\frac{|x|}{\nu})U_{x^1,\va}^{2^*_s}+k2^*_s\int_{\Omega_1}K(\frac{|x|}{\nu})\sum_{i=2}^kU_{x^1,\va}^{2^*_s-1}U_{x^i,\va}\\
&& +k\left\{\begin{array}{ll}
         O\Big(\ds\int_{\Omega_1}U_{x^1,\va}^{\frac{2^*_s}{2}}(\sum_{i=2}^kU_{x^i,\va})^{\frac{2^*_s}{2}}\Big), & \text{if}\  2<2^*_s<3,\vspace{3mm}\\
       O\Big(\ds\int_{\Omega_1}U_{x^1,\va}^{2^*_s-2}(\sum_{i=2}^kU_{x^i,\va})^2\Big), & \text{if}\ 2^*_s\geq3.
    \end{array}\right.
\end{eqnarray*}
For $x\in \Omega_1$, $|x-x^i|\geq \frac12|x^i-x^1|$ and $|x-x^i|\geq |x-x^1|,$ we obtain
\begin{equation*}\begin{split}
\sum_{i=2}^kU_{x^i,\va}&\leq
C\sum_{i=2}^k\frac{1}{(1+|x-x^i|)^{N-2s-\kappa}}\frac{1}{(1+|x-x^i|)^{\kappa}}\\
 &\leq
C\sum_{i=2}^k
\frac{1}{|x^1-x^i|^{N-2s-\kappa}}\frac{1}{(1+|x-x^1|)^{\kappa}},\\
\end{split}
\end{equation*}
where $s<\kappa<\min\{\frac43 s,\ \frac{N-2s}{2}\}$.
Hence, we get
\begin{eqnarray*}
&&\int_{\Omega_1}U_{x^1,\va}^{\frac{2^*_s}{2}}(\sum_{i=2}^kU_{x^i,\va})^{\frac{2^*_s}{2}}\\
&\leq&C\int_{\Omega_1}\frac{1}{(1+|x-x^1|)^{\frac{(N-2s)2^*_s}{2}}}\Big(\sum_{i=2}^k\frac{1}{|x^i-x^1|^{N-2s-\kappa}}\Big)^\frac{2^*_s}{2}
\frac{1}{(1+|x-x^1|)^{\frac{2^*_s}{2}\kappa}}\\
&=&C\Big(\sum_{i=2}^k\frac{1}{|x^i-x^1|^{N-2s-\kappa}}\Big)^\frac{2^*_s}{2}\int_{\Omega_1}\frac{1}{(1+|x-x^1|)^{\frac{(N-2s)2^*_s}{2}+\frac{2^*_s\kappa}{2}}}\\
&\le&C\Big(\sum_{i=2}^k\frac{1}{|x^i-x^1|^{N-2s-\kappa}}\Big)^\frac{2^*_s}{2}\le C\Big(\frac{k}{\nu}\Big)^{N-2s+\sigma}
\end{eqnarray*}
and
\begin{eqnarray*}
&&\int_{\Omega_1}U_{x^1,\va}^{2^*_s-2}(\sum_{i=2}^kU_{x^i,\va})^2\\
&\leq&C\int_{\Omega_1}\frac{1}{(1+|x-x^1|)^{(N-2s)(2^*_s-2)}}\Big(\sum_{i=2}^k\frac{1}{|x^i-x^1|^{N-2s-\kappa}}\Big)^2\frac{1}{(1+|x-x^1|)^{2\kappa}}\\
&=&C\Big(\sum_{i=2}^k\frac{1}{|x^i-x^1|^{N-2s-\kappa}}\Big)^2\int_{\Omega_1}\frac{1}{(1+|x-x^1|)^{(N-2s)(2^*_s-2)+2\kappa}}\\
&=&C\Big(\sum_{i=2}^k\frac{1}{|x^i-x^1|^{N-2s-\kappa}}\Big)^2\le  C\Big(\frac{k}{\nu}\Big)^{N-2s+\sigma}.
\end{eqnarray*}
On the other hand,
\begin{equation*}\begin{split}
\int_{\Omega_1}K(\frac{|x|}{\nu})U_{x^1,\va}^{2^*_s-1}\sum_{i=2}^kU_{x^i,\va}&=\int_{\Omega_1}U_{x^1,\va}^{2^*_s-1}\sum_{i=2}^kU_{x^i,\va}+\int_{\Omega_1}(K(\frac{|x|}{\nu})-1)U_{x^1,\va}^{2^*_s-1}\sum_{i=2}^kU_{x^i,\va}.
\end{split}
\end{equation*}
But, from Lemmas~\ref{lmA.1} and \ref{A.2},
\begin{eqnarray*}
&&\int_{\Omega_1}U_{x^1,\va}^{2^*_s-1}\sum_{i=2}^kU_{x^i,\va}=\int_{\r^N}U_{x^1,\va}^{2^*_s-1}\sum_{i=2}^kU_{x^i,\va}
-\int_{\r^N\setminus\Omega_1}U_{x^1,\va}^{2^*_s-1}\sum_{i=2}^kU_{x^i,\va}\\
&=&\int_{\r^N}U_{x^1,\va}^{2^*_s-1}\sum_{i=2}^kU_{x^i,\va}
+O\Big(\Big(\frac{k}{r}\Big)^\sigma\int_{\r^N\setminus\Omega_1}U_{x^1,\va}^{2^*_s-1-\sigma}\sum_{i=2}^kU_{x^i,\va}\Big)\\
&=&\int_{\r^N}U_{x^1,\va}^{2^*_s-1}\sum_{i=2}^kU_{x^i,\va}+O\Big(\Big(\frac{k}{\nu r_0}\Big)^\sigma\sum_{i=2}^k\frac{1}{|x^i-x^1|^{N-2s}}
\int_{\r^N\setminus\Omega_1}\Big(U_{x^1,\va}^{2^*_s-1-\sigma}+U^{2^*_s-1-\sigma}_{x^i,\va}\Big)\Big)\\
&=&\sum_{i=2}^k\frac{B_0'}{\va^{N-2s}|x^1-x^i|^{N-2s}}+O\Big(\Big(\frac{k}{\nu}\Big)^{N-2s+\sigma}\Big),
\end{eqnarray*}
where $0<\sigma<\frac{2s}{N-2s}$.
Moreover, similarly,
\begin{eqnarray*}
&&\int_{\Omega_1}\Big|K(\frac{|x|}{\nu})-1\Big|U_{x^1,\va}^{2^*_s-1}\sum_{i=2}^kU_{x^i,\va}\\
&=&\int_{\r^N}\Big|K(\frac{|x|}{\nu})-1\Big|U_{x^1,\va}^{2^*_s-1}\sum_{i=2}^kU_{x^i,\va}
-\int_{\r^N\setminus\Omega_1}\Big|K(\frac{|x|}{\nu})-1\Big|U_{x^1,\va}^{2^*_s-1}\sum_{i=2}^kU_{x^i,\va}\\
&\le&\int_{B_{\frac{|x^1|}{2}}(x^1)}\Big|K(\frac{|x|}{\nu})-1\Big|U_{x^1,\va}^{2^*_s-1}\sum_{i=2}^kU_{x^i,\va}+C\int_{\r^N\setminus B_{\frac{|x^1|}{2}}(x^1)}U_{x^1,\va}^{2^*_s-1}\sum_{i=2}^kU_{x^i,\va}
+O\Big(\Big(\frac{k}{r}\Big)^{N-2s+\sigma}\Big)\\
&\le&\frac{C}{\nu^m}\int_{\r^N}U_{x^1,\va}^{2^*_s-1}\sum_{i=2}^kU_{x^i,\va}+O\Big((\frac{k}{\nu})^\sigma\sum_{i=2}^k\frac{1}{|x^i-x^1|^{N-2s}}
\int_{\r^N\setminus B_{\frac{|x^1|}{2}}(x^1)}\Big(U_{x^1,\va}^{2^*_s-1-\sigma}+U^{2^*_s-1-\sigma}_{x^i,\va}\Big)\Big)\\
&&+O\Big(\Big(\frac{k}{\nu}\Big)^{N-2s+\sigma}\Big)\\
&=&O\Big(\Big(\frac{k}{\nu}\Big)^{N-2s+\sigma}+\frac{1}{\nu^{m+\sigma}}\Big).
\end{eqnarray*}
Hence,
\begin{equation*}\begin{split}
\int_{\Omega_1}K(\frac{|x|}{\nu})U_{x^1,\va}^{2^*_s-1}\sum_{i=2}^kU_{x^i,\va}
=\sum_{j=2}^k\frac{B'_0}{\va^{N-2s}|x^1-x^i|^{N-2s}}+O\Big(\Big(\frac{k}{\nu}\Big)^{N-2s+\sigma}+\frac{1}{\nu^{m+\sigma}}\Big).
\end{split}
\end{equation*}
Finally,
\begin{eqnarray*}
&&\int_{\Omega_1}K(\frac{|x|}{\nu})U_{x^1,\va}^{2^*_s}\\
&=&\int_{\Omega_1}U_{x^1,\va}^{2^*_s}-\frac{C}{\nu^m}\int_{\Omega_1}\big||x|-\nu r_0\big|^mU_{x^1,\va}^{2^*_s}+O\Big(\nu^{-m-\theta}\int_{\Omega_1}\big||x|-\nu r_0\big|^{m+\theta}U_{x^1,\va}^{2^*_s}\Big)\\
&=&\int_{\r^N}U_{0,1}^{2^*_s}-\frac{C}{\nu^m}\int_{\r^N}\big||x-x^1|-\nu r_0\big|^mU_{0,\va}^{2^*_s}+O(\frac1{\nu^{m+\theta}}).\\
\end{eqnarray*}
But
$$
\frac{C}{\nu^m}\int_{\r^N\setminus B_{\frac{|x^1|}2}(0)}\big||x-x^1|-\nu r_0\big|^mU_{0,\va}^{2^*_s}
\leq C \int_{\r^N\setminus B_{\frac{|x^1|}2}(0)}\Big(\frac{|x|^m}{\nu^m}+1\Big)U_{0,\va}^{2^*_s}\leq \frac{C}{\nu^N}.
$$
On the other hand, if $x\in B_{\frac{|x^1|}2}(0),$ $x=(x_1,x^*),$  then $|x^1|-x_1\geq \frac{|x^1|}{2}>0.$ We know
$$
|x-x^1|=|x^1|-x_1+O\Big(\frac{|x^*|^2}{|x^1|-x_1}\Big)=|x^1|-x_1+O\Big(\frac{|x^*|^2}{|x^1|}\Big).
$$
So,
\begin{eqnarray*}
\big||x-x^1|-\nu r_0\big|^m&=&\Big||x^1|-x_1+O\Big(\frac{|x^*|^2}{|x^1|}\Big)-\nu r_0\Big|^m\\
&=& |x_1|^m+m|x_1|^{m-2}x_1\Big(\nu r_0-|x^1|+O\Big(\frac{|x|^2}{|x^1|}\Big)\Big)\\
&&+\frac12m(m-1)|x_1|^{m-2}\Big(\nu r_0-|x^1|+O\Big(\frac{|x|^2}{|x^1|}\Big)\Big)^2\\
&&+O\Big(\Big(\nu r_0-|x^1|+O\Big(\frac{|x|^2}{|x^1|}\Big)\Big)^{2+\sigma}\Big),
\end{eqnarray*}
and using
$$
\int_{B_{\frac{|x^1|}2}(0)}|x_1|^{m-2}x_1U_{0,\va}^{2^*_s}=0,
$$
we get
\begin{eqnarray*}
\int_{ B_{\frac{|x^1|}2}(0)}\big||x-x^1|-\nu r_0\big|^mU_{0,\va}^{2^*_s}&=&  \int_{\r^N}|x_1|^mU_{0,\va}^{2^*_s}+\frac12m(m-1)\int_{\r^N}|x_1|^{m-2}U_{0,\va}^{2^*_s}(\nu r_0-|x^1|)^2\\
&&+O\Big(\big|\nu r_0-|x^1|\big|^{2+\sigma}\Big).
\end{eqnarray*}
Therefore,
\begin{equation}\label{A.3}\begin{split}
\int_{\r^N}K(\frac{|x|}{\nu})U_{r,\va}^{2^*_s}=& k\Big(\int_{\r^N}U_{0,1}^{2^*_s}-\frac{c_0}{\va^m \nu^m}\int_{\r^N}|x_1|^mU_{0,1}^{2^*_s}\\
&-\frac{c_0}{\va^{m-2} \nu^m}\frac12m(m-1)\int_{\r^N}|x_1|^{m-2}U_{0,\va}^{2^*_s}(\nu r_0-|x^1|)^2\\
&+2^*_s\sum_{i=2}^k\frac{B_0}{\va^{N-2s}|x^1-x^i|^{N-2s}}+O(\frac{1}{\nu^{m+\sigma}})\Big).
\end{split}
\end{equation}
Now, inserting \eqref{A.1}--\eqref{A.3} into $I(U_{r,\va})$, we complete
the proof.
\end{proof}
Similar to  Proposition \ref{propA.1} and Proposition A.2 in \cite{wy1}, we have
\begin{proposition}\label{propA.2}
\begin{equation*}\begin{split}
\frac{\partial I(U_{r,\va})}{\partial \va}=&k\Big(-\frac{m B_0}{\va^{m+1}\nu^m}+\sum_{i=2}^k\frac{B_2(N-2s)}{\va^{N-2s+1}|x^1-x^i|^{N-2s}}+O\big(\frac{1}{\nu^{m+\sigma}}+\frac{1}{\nu^{m}}\big|\nu r_0-|x^1|\big|^{2}\big)\Big)
\end{split}
\end{equation*}
 where
$B_0\in [C_1,C_2],\ 0<C_1<C_2,$ $A=\frac{s}{N}\int_{\r^N}U_{0,1}^{2^*_s},\  B_i,\ C_i,\ i=1,2,$ are some constants, and $r=|x^1|$.
\end{proposition}

Applying the same argument as in the proof of Proposition A.4 in \cite{lpy}, we also can prove the following result.

\begin{proposition}\label{propA.3} We have
\begin{equation*}\begin{split}
I(\bar U_{r,\va})=&k\Big(A-\frac{B'_0}{\va^{m}\nu^m}-\frac{B'_1}{\va^{m-2}\nu^m}(\nu r_0-r)^2+\sum_{i=2}^{2k}\frac{B'_2k^{N-2s}}{\va^{N-2s}r^{N-2s}}\\
&\quad+ O\big(\frac{1}{\nu^{m+\sigma}}+\frac{1}{\nu^{m}}\big|\nu r_0-|\bar{x}^1|\big|^{2+\sigma}+\frac{k}{r^{N-2s}}\big)\Big)\\
\end{split}
\end{equation*}
 where
$B_0'\in [C_1',C_2'],\ 0<C_1'<C_2',$ $A=\frac{s}{N}\int_{\r^N}U_{0,1}^{2^*_s},\  B_i',\ C_i',\ i=1,2,$ are some constants, and $r=|x^1|$.
\end{proposition}
\begin{proof}
Similar to  Proposition \ref{propA.1}, we have
\begin{equation*}\begin{split}
I(\bar U_{r,\va})=&k\Big(A-\frac{B_0}{\va^{m}\nu^m}-\frac{B_1}{\va^{m-2}\nu^m}(\nu r_0-r)^2-\sum_{i=2}^{2k}\frac{(-1)^{j-1}B_2}{\va^{N-2s}|\bar x^1-\bar x^i|^{N-2s}}\\
&+O\big(\frac{1}{\nu^{m+\sigma}}+\frac{1}{\nu^{m}}\big|\nu r_0-|\bar x^1|\big|^{2+\sigma}\big)\Big)
\end{split}
\end{equation*}
 where
$B_0\in [C_1,C_2],\ 0<C_1<C_2,$ $A=\frac{s}{N}\int_{\r^N}U_{0,1}^{2^*_s},\  B_i,\ C_i,\ i=1,2,$ are some constants, and $r=|\bar x^1|$.
Therefore, it is sufficient to prove that
$$
\sum_{j=2}^{2k}\frac{(-1)^{j}}{|\bar x^1-\bar x^j|^{N-2s}}=\frac{\bar B'_0
k^{N-2s}}{r^{N-2s}}+O\Big(\frac{k}{r^{N-2s}}\Big),
$$
for some $\bar B'_0>0$, which can be easily verified by using the following facts
\begin{eqnarray*}
\sum_{i=2}^{2k}\frac{1}{|x^i-x^1|^{N-2s}}&=&
\left\{\begin{array}{ll}
2\sum_{i=2}^{k}\frac{(-1)^i}{|\bar x^1-\bar x^i|^{N-2s}}+\frac{1}{(2|\bar x^1|)^{N-2s}}, & \text{if}\  k\ \text{is \ odd}\vspace{2mm}\\
       2\sum_{i=2}^{k}\frac{(-1)^i}{|\bar x^1-\bar x^i|^{N-2s}}-\frac{1}{(2|\bar x^1|)^{N-2s}}, & \text{if}\ k\ \text{is \ even}
    \end{array}\right.\\
& =& \left\{\begin{array}{ll}
        \frac{2}{(2|\bar x^1|)^{N-2s}}\sum_{i=2}^{k}\frac{(-1)^i}{(\sin\frac{(i-1)\pi}{2k})^{N-2s}}+\frac{1}{(2|\bar x^1|)^{N-2s}}, & \text{if}\  k\ \text{is \ odd}\vspace{2mm}\\
       \frac{2}{(2|\bar x^1|)^{N-2s}}\sum_{i=2}^{k}\frac{(-1)^i}{(\sin\frac{(i-1)\pi}{2k})^{N-2s}}-\frac{1}{(2|\bar x^1|)^{N-2s}}, & \text{if}\ k\ \text{is \ even}
    \end{array}\right.
\end{eqnarray*}
and
$$
0<c'\leq\frac{\sin\frac{(i-1)\pi}{2k}}{\frac{(i-1)\pi}{2k}}\leq C'',\
\ i=2,\cdots,k.
$$
\end{proof}

Similar to  Proposition \ref{propA.2} and \ref{propA.3}, we have
\begin{proposition}\label{propA.4}
\begin{equation*}\begin{split}
\frac{\partial I(U_{r,\va})}{\partial \va}=&k\Big(\frac{m B'_0}{\va^{m+1}\nu^m}-\frac{B'_2(N-2s)}{\va^{N-2s+1}r^{N-2s}}+O\big(\frac{1}{\nu^{m+\sigma}}+\frac{1}{\nu^{m}}\big|\nu r_0-|\bar x^1|\big|^{2}+\frac{k}{r^{N-2s}}\big)\Big)
\end{split}
\end{equation*}
 where
$B'_0\in [C'_1,C'_2],\ 0<C'_1<C'_2,$ $A=\frac{s}{N}\int_{\r^N}U_{0,1}^{2^*_s},\  B'_i,\ C'_i,\ i=1,2,$ are some constants, and $r=|\bar x^1|$.
\end{proposition}


\begin{thebibliography}{99}
{\footnotesize


\bibitem{aap} A. Ambrosetti, J. Garcia Azorero, I. Peral, Perturbation of $-\Delta u = u^{\frac{N+2}{N-2}}$, the scalar curvature problem in $\r^N$ and
related topics, J. Funct. Anal. 165 (1999) 117--149.

 \bibitem{a} D. Applebaum,  L\'{e}vy processes and stochastic calculus, Second edition, Cambridge Studies in Advanced Mathematics, 116. Cambridge
University Press, Cambridge, 2009.

\bibitem{bcps} B. Barrios, E. Colorado, A. de Pabloc, U. S\'{a}nchezc, On some critical problems for the fractional Laplacian
operator, J. Differerntial Equations 252 (2012) 6133--6162.

\bibitem{b} J. Bertoin,  L\'{e}vy processes, Cambridge Tracts in Mathematics, 121.
Cambridge University Press, Cambridge, 1996.

\bibitem{cs} L. Caffarelli, L. Silvestre, An extension problem related to the fractional Laplacian,
Comm. Partial Differential Equations 32 (2007) 1245--1260.

\bibitem{cl} E.A. Carlen, M. Loss, Extremals of functionals with competing symmetries, J. Funct. Anal.
88 (1990)  437--456.

\bibitem{cny} D. Cao, E. Noussair, S. Yan, On the scalar curvature equation $-\Delta u = (1 + \varepsilon K)u^{\frac{N+2}{N-2}}$ in $\r^N$, Calc. Var.
Partial Differential Equations 15 (2002) 403--419.

\bibitem{clo} W. Chen, C. Li, B. Ou, Classification of solutions for an integral equation, Comm. Pure
Appl. Math. 59 (2006) 330--343.

\bibitem{cwy} W. Chen, J. Wei, S. Yan, Infinitely many solutions for the Schr\"{o}dinger equations in $\r^N$ with critical growth,
J. Differential Equations 252 (2012) 2425--2447.

\bibitem{cz} G. Chen, Y. Zhang, Concentration phenomenon for
fractionsl nonlinear Schr\"{o}dinger equations, arXiv: 1305.4426.

\bibitem{cp} E. Colorado, I. Peral, Semilinear elliptic problems with mixed Dirichlet-
Neumann boundary conditions. J. Funct. Anal. 199 (2003) 468--507.

\bibitem{dps} J. D\'{a}vila, M. Del Pino, Y. Sire, Non degeneracy of the bubble in the critial case for non local equation,
  Proc. Amer. Math. Soc. 141 (2) (2013) 3865--3870.

\bibitem{dpw} J. D\'{a}vila, M. Del Pino, J. Wei, Concentrating
standing waves for fractional nonlinear Schr\"{o}dinger equation, J.
Differerntial Equations, 256 (2014) 858--892.

\bibitem{dfm}
M. del Pino, P. Felmer, M. Musso, Two-bubble solutions in the super-critical Bahri-Coron's problem,
Calc. Var. Partial Differential Equations 16 (2003) 113--145.

\bibitem{dpv} S. Dipierro, G. Palatucci, E. Valdinoci, Existence and symmetry results for a Schr\"{o}dinger type problem involving the fractional Laplacian,
Matematiche, 68 (2013) 201--216.



\bibitem{fl1} R.L. Frank, E.H. Lieb, Inversion positivity and the sharp Hardy-Littlewood-Sobolev inequality,
Calc. Var. Partial Differential Equations 39 (2010) 85--99.

\bibitem{fl2} R.L. Frank, E.H. Lieb, A new, rearrangement-free proof of the sharp Hardy-Littlewood-
Sobolev inequality. Spectral Theory, Function Spaces and Inequalities,
55¨C67, Oper. Theory Adv. Appl., 219, Birkh\"{a}user/Springer Basel AG, Basel, 2012.

\bibitem{gt} D. Gilbarg,  N.S. Trudinger, Elliptic partial differential equations of
second order. Reprint of the 1998 edition. Classics in Mathematics. Springer-Verlag, Berlin, 2001.

\bibitem{li1} Y.Y. Li, On $-\Delta u = K(x)u^5$ in $R^3$, Comm. Pure Appl. Math. 46 (1993) 303--340.

\bibitem{li} Y.Y. Li, Remark on some conformally invariant integral equations: the method of moving
spheres, J. Eur. Math. Soc. (JEMS) 6 (2004) 153--180.



\bibitem{lz} Y.Y. Li, M. Zhu, Uniqueness theorems through the method of moving spheres, Duke Math.
J. 80 (1995) 383--417.

\bibitem{l} E.H. Lieb, Sharp constants in the Hardy-Littlewood-Sobolev and related inequalities, Ann. of
Math, (2) 118 (1983) 349--374.

\bibitem{lpy}
W. Long, S. Peng, J. Yang, Infinitely many positive solutions and sign-changing solutions for  nonlinear  fractional
scalar field equations, Discret. Contin. Dynam. Syst. 36 (2016) 917--939.



\bibitem{sv} R. Servadei, E. Valdinoci, Variational methods for non-local operators of elliptic type,
Discrete Contin. Dyn. Syst. 33 (2013) 2105--2137.
\bibitem{sv1} R. Servadei, E. Valdinoci, A Brezis-Nirenberg result for non-local critical equations in low dimension,
Commun. Pure Appl. Anal. 12 (2013) 2445--2464.

\bibitem{sz} X. Shang, J. Zhang, Ground states for fractional Schr\"{o}dinger equations with critical growth, Nolinearity 27 (2014) 187--207.

\bibitem{stein} E.M. Stein, Singular integrals and differentiability properties of functions,
Princeton Mathematical Series, No. 30 Princeton University Press, Princeton, N.J. 1970 xiv+290 pp. (Reviewer: R. E. Edwards) 46.38 (26.00).

\bibitem{s} M. Struwe, Variational methods. Applications to nonlinear partial differential
equations and Hamiltonian systems, Springer-Verlag, Berlin,
2000.

\bibitem{p} M. Picone, Sui valori eccezionali di un paramtro da cui dipende una
equazione differenziale lineare ordinaria del secondo ordine, Ann.
Scuola. Norm. Pisa., 11 (1910) 1--144.

\bibitem{t} J. Tan, The Brezis-Nirenberg type problem involving the square root of the Laplacian, Calc.
Var. Partial Differential Equations 42 (2012) 21--41.


\bibitem{wy1} J. Wei, S. Yan, Infinite many positive solutions for
the prescribed scalar curvature problem on $\mathbb{S}^N$, J. Funct.
Anal. 258 (2010) 3048--3081.

\bibitem{wy2} J. Wei, S. Yan, Infinitely many positive solutions for an elliptic problem with critical or supercritical growth,
J. Math. Pures Appl. 96 (2011) 307--333.

\bibitem{wy3} J. Wei, S. Yan, Infinitely many nonradial solutions for the H\'{e}non equation with critical growth,
 Rev. Mat. Iberoam. 29 (2013) 997--1020.

\bibitem{y} S. Yan, Concentration of solutions for the scalar curvature equation on $\r^N$, J. Differential Equations 163 (2000)
239--264.

}

\end{thebibliography}
\end{document}